\documentclass[11pt]{amsart}

\usepackage{amsmath}
\usepackage{amssymb,amsthm,amsfonts}
\usepackage{amsrefs}
\usepackage[mathscr]{euscript}
\usepackage[margin=1.8cm]{geometry}
\usepackage{graphicx}
\usepackage{wrapfig}
\usepackage[usenames,dvipsnames,svgnames,table]{xcolor}

\renewcommand{\mathcal}{\mathscr}

\def\R {\mathbb{R}}

\def\N {\mathbb{N}}

\def\Z {\mathbb{Z}}

\renewcommand{\epsilon}{\varepsilon}
\newcommand{\eps}{\varepsilon}

\renewcommand{\leq}{\leqslant}
\renewcommand{\le}{\leqslant}
\renewcommand{\geq}{\geqslant}
\renewcommand{\ge}{\geqslant}

\newcommand{\osc}[1]{{\underset{#1}{\operatorname{osc}\,}}}

\newtheorem{proposition}{Proposition}[section]
\newtheorem{theorem}[proposition]{Theorem}
\newtheorem{corollary}[proposition]{Corollary}
\newtheorem{lemma}[proposition]{Lemma}
\theoremstyle{definition}

\newtheorem{remark}[proposition]{Remark}
\numberwithin{equation}{section}

\allowdisplaybreaks

\begin{document}
\title[Heteroclinic connections and Dirichlet problems]{Heteroclinic connections and Dirichlet problems \\
for a nonlocal functional of oscillation type}

\thanks{SD and EV are supported by the Australian Research Council
Discovery Project ``N.E.W. Nonlocal Equations at Work'' DP170104880. 
MN is partially supported by the University of Pisa Project 
PRA 2017 ``Problemi di ottimizzazione e di evoluzione in ambito variazionale''.
The authors are members of INdAM--GNAMPA}

\thanks{
{\em Annalisa Cesaroni}:
Dipartimento di Scienze Statistiche,
Universit\`a di Padova, Via Battisti 241/243, 35121 Padova, Italy. {\tt annalisa.cesaroni@unipd.it}
\\
{\em Serena Dipierro}:
Department of Mathematics
and Statistics,
University of Western Australia,
35 Stirling Hwy, Crawley WA 6009, Australia.
{\tt serena.dipierro@uwa.edu.au}\\
{\em Matteo Novaga}: Dipartimento di Matematica,
Universit\`a di Pisa, 
Largo Pontecorvo 5, 56127 Pisa,
Italy. {\tt matteo.novaga@unipi.it}\\
{\em Enrico Valdinoci}:
Department of Mathematics
and Statistics,
University of Western Australia,
35 Stirling Hwy, Crawley WA 6009, Australia,
Dipartimento di Matematica, Universit\`a di Milano,
Via Saldini 50, 20133 Milan, Italy. {\tt enrico@math.utexas.edu}}

\author[A. Cesaroni]{Annalisa Cesaroni}
\author[S. Dipierro]{Serena Dipierro}
\author[M. Novaga]{Matteo Novaga}
\author[E. Valdinoci]{Enrico Valdinoci}

\begin{abstract}
We consider an energy functional combining
the square of the local oscillation of a one--dimensional
function with a double well potential.
We establish the existence of minimal heteroclinic solutions
connecting the two wells of the potential.

This existence result cannot be accomplished by standard methods,
due to the lack of compactness properties.

In addition, we investigate the main properties of these
heteroclinic connections. We show that these minimizers
are monotone, and therefore they satisfy a suitable Euler-Lagrange equation.

We also prove that,
differently from the classical cases arising in ordinary differential
equations, in this context the heteroclinic connections
are not necessarily smooth, and not even continuous
(in fact, they can be piecewise constant). Also, we show that
heteroclinics are not necessarily unique up
to a translation, which is also in contrast with the classical setting.

Furthermore, we investigate the associated Dirichlet problem,
studying existence, uniqueness and partial regularity properties,
providing explicit solutions in terms of the external data and of the forcing
source, and exhibiting an example of discontinuous solution.
\end{abstract}

\maketitle

\tableofcontents

\section{Introduction}

One of the most classical problems in ordinary differential equations
consists in the study of second order equations coming from mechanical systems
having a Hamiltonian structure. For instance, one can consider the simple
one--dimensional case in which the Hamiltonian has the form
$$ H(p,q)=\frac{p^2}2-W(q)\,,\qquad p,q\in\R,$$
giving rise to the system of equations
$$ \begin{cases}
\dot q=\partial_p H=p,\\
\dot p=-\partial_q H=W'(q).\end{cases}$$
Noticing that~$\ddot q=\dot p$,
this system of ordinary differential equations reduces
to the single second order equation
\begin{equation}\label{AKSMSJsdfSJ}
\ddot q=W'(q).
\end{equation}
Equation~\eqref{AKSMSJsdfSJ} has a variational structure, coming from
the action functional
\begin{equation}\label{action} \int_\R \frac{\dot q^2(t)}{2}+W(q(t))\,dt,\end{equation}
namely minimizers (or, more generally, critical points) of this functional
satisfy~\eqref{AKSMSJsdfSJ}:
in jargon, one says that~\eqref{AKSMSJsdfSJ}
is the Euler-Lagrange equation associated to the functional in~\eqref{action}.

A typical and concrete example of this setting
is given by the equation of the pendulum:
in this case, setting the Lyapunov stable
equilibrium of the pendulum at~$q=0$
and the Lyapunov
unstable\footnote{We remark that the
Lyapunov stable equilibrium of the pendulum is variationally unstable, namely
the second derivative of the action functional is negatively defined. Viceversa,
the Lyapunov unstable equilibria of the pendulum are variationally stable, since
the second derivative of the action functional is positively defined. The terminology
related to Lyapunov stability is perhaps more common in the dynamical systems community,
while the one dealing with variational stability is often adopted in the calculus
of variations.}
ones at~$q=\pm1$, and considering unit gravity for simplicity,
one can take
\begin{equation}\label{WW1}
W(q)=\frac{1+\cos(\pi q)}{\pi}\end{equation}
in~\eqref{AKSMSJsdfSJ} and obtain the equation
\begin{equation}\label{AKSMSJsdfSJ2a} \ddot q=-\sin(\pi q).\end{equation}
An interesting analogue of~\eqref{AKSMSJsdfSJ} in partial differential
equations
arises in the study of phase coexistence models, and in particular
in the analysis of the Allen-Cahn equation
\begin{equation}\label{AKSMSJsdfSJ2}
\Delta u+u-u^3=0.
\end{equation}
If one considers the one-dimensional case,
with the choice
\begin{equation}\label{WW2}W(u)=\frac{(1-u^2)^2}{4},\end{equation}
then~\eqref{AKSMSJsdfSJ2} can be also reduced to~\eqref{AKSMSJsdfSJ}.
\medskip

A well established topic for the dynamical systems
as in~\eqref{AKSMSJsdfSJ2a} is the search for heteroclinic orbits,
that are orbits which connect two (Lyapunov unstable) equilibria.
These solutions have the special feature of separating the phase
space into regions in which solutions exhibit different topological behaviors
(e.g. oscillations versus librations), and, in higher dimensions,
they provide the essential building block to chaos.

The analogue of such heteroclinic connections for
the phase coexistence problems in~\eqref{AKSMSJsdfSJ2}
provides a transition layer connecting two (variationally stable)
pure phases of the system. In higher dimensions,
these solutions constitute the cornerstone to describe at a large scale
the phase separation, as well as the phase parameter in dependence
to the distance from the interface.\medskip

In this article we explore a brand new line of investigation
focused on a nonlocal analogue of~\eqref{AKSMSJsdfSJ},
in which the second derivative is replaced by a finite difference.
More concretely, we will consider a functional similar to that in~\eqref{action},
but in which the derivative is replaced by an oscillation term. We recall that other
nonlocal analogues 
of~\eqref{AKSMSJsdfSJ} have been considered in the literature, mainly replacing the 
second derivative with a fractional second derivative, see \cite{psv, a, cs, cmy, JAaasdfgN}.
Other lines of investigation took into account the case in which
the second derivative is replaced by a quadratic interaction with an integrable kernel,
see~\cite{MR1612250} and the references therein. \medskip

The interest in this problem combines perspectives
in pure and in applied mathematics. Indeed, from the theoretical
point of view, nonlocal functionals typically exhibit a number of
{\em novel features} that are worth exploring and provide {\em several conceptual
difficulties} that are completely new with respect to the classical cases.
On the other hand, in terms of applications,
nonlocal functionals can capture {\em original and interesting phenomena}
that cannot be described by the classical models.\medskip

In our framework, in particular, we take into account a nonlocal interaction
which is not scale invariant. This type of nonlocal structures is closely
related to several geometric motions that have been recently studied
both for their analytic interest in the calculus of variations
and for their concrete applicability in situations in which
detecting different scales allows the preservation of details and
irregularities in the process of removing white noises
(e.g. in the digitalization of fingerprints,
in which one wants to improve
the quality of the image without losing relevant features at small scales).
We refer in particular to~\cite{MR2728706, cdnv, cdnv2, cn, MR2655948, MR3187918, MR3023439, MR3401008, dnv}
for several recent contributions in the theoretical and applied analysis
of nonlocal problems without scale invariance.\medskip

While the previous literature mostly focuses on geometric
evolution equations, viscosity solutions, perimeter type problems
and questions arising in the calculus of variations,
in this paper we aim at investigating the existence and basic properties of heteroclinic minimizers
for nonlocal problems with lack of scale invariance.

Since this topic of research is completely new, we will need to introduce
the necessary methodology from scratch. In particular,
one cannot rely on standard methods, since:
\begin{itemize}
\item the problems taken into account do not possess standard
compactness properties,
\item the functional to minimize cannot be easily differentiated,
\item the solutions found need not to be (and in general are not) regular,
\item the solutions found need not to be (and in general are not) unique.
\end{itemize}
In the rest of the introduction, we give formal statements concerning the
mathematical setting in which we work and we present our main results,
regarding the existence of the heteroclinic connections,
their monotonicity properties, the Euler-Lagrange equation that they satisfy,
and their lack of regularity and uniqueness. Then, we take into account the Dirichlet
problem, obtaining explicit solutions and optimal oscillation bounds.

\subsection{Main assumptions} 
Given an interval~$I\subset\R$, we consider the oscillation of a function~$u\in L^\infty(I)$,
defined as
\begin{equation}\label{1.1bis}
\osc{I}u:=\sup_I u-\inf_I u.\end{equation}
Given~$r>0$, $a<b$, with~$b-a>2r$,
and~$W\in C(\R)$, we consider the energy functional
\begin{equation}\label{CALE}
{\mathcal{E}}_{(a,b)}(u):=\frac1{2r^2}\int_a^b \left(\osc{(x-r,x+r)} u\right)^2\,dx+\int_{a}^{b} W(u(x))\,dx.\end{equation}
As customary, we say that~$u\in L^\infty_{\rm loc}(\R)$ is a {\it local minimizer}
of~${\mathcal{E}}$
if, for any~$a<b$ and any~$v\in L^\infty_{\rm loc}(\R)$ such that~$u=v$
outside~$[a+r,b-r]$, we have that
$$ {\mathcal{E}}_{(a,b)}(u)\le{\mathcal{E}}_{(a,b)}(v).$$
Notice that, due to the nonlocal character of the oscillation functional, we require the competitor $v$ to coincide with the minimizer $u$ outside~$[a+r,b-r]$ instead of~$[a,b]$ (see \cite{cdnv2} for a discussion of this issue).

\smallskip

We shall assume the following structural conditions on the potential $W$: 
\begin{equation}\label{DOUBLEW} \begin{cases} 
W\in C(\R),\qquad  W(-1)=W(1)=0<W(t)\quad \mbox{for
all~$t\in\R\setminus\{-1,1\}$},\\ 
{\mbox{$W$ is strictly decreasing in~$(-\infty,-1)$ and
strictly increasing in~$(1,+\infty)$,}}
\\
{\mbox{$W$ is an even   function in $[-1,1]$ and has a
unique local maximum at $t=0$,}}\end{cases}
\end{equation}
and we denote 
\begin{equation}\label{cw}
c_W:=\int_{-1}^1 W(s)\,ds>0.
\end{equation}
%%	We remark that the third assumption on $W$ in \eqref{DOUBLEW}
%%	is equivalent to require that $W$ is nonincreasing on $[0,1]$
%%	and nondecreasing on $[-1,0]$.
The last assumption in~\eqref{DOUBLEW}
can be slightly generalized by simply assuming that~$W$ has a
unique local maximum in~$[-1,1]$, with 
some minor technical adaptations of our arguments. 
We point out that  condition \eqref{DOUBLEW} is satisfied
by the standard ``double well'' potentials, e.g. the ones in~\eqref{WW1}
and~\eqref{WW2}.\medskip

In the forthcoming Subsections \ref{MAINRE1}, \ref{MAINRE2} and~\ref{IAKNSHDKNDN}
we give precise statements of our main results
concerning the existence, possible uniqueness,
and geometric properties of
the minimizers of the functional in~\eqref{CALE},
specifically focused on heteroclinic connections, that, in our
setting, are critical points of the functional which connect the
two equilibria~$-1$ and~$1$. The Dirichlet problem associated to~\eqref{CALE}
(when restricted to monotone functions) will be described in detail in Subsection~\ref{THEDIRIC}.

\subsection{Existence of minimal heteroclinic connections}\label{MAINRE1} 

Now we discuss the construction of local minimizers
to~\eqref{CALE} which connect the two stable equilibria $-1$ and $1$. 
Our main result on this topic goes as follows:

\begin{theorem}[Existence of minimal heteroclinic connections] \label{exthm}
Assume that~\eqref{DOUBLEW} holds true. 
Then, there exists  a local minimizer~$u\in L^\infty(\R)$ to \eqref{CALE}  such that  $u$ is monotone nondecreasing and   satisfies 
\begin{equation}\label{LIMITS}
\lim_{x\to\pm\infty} u(x)=\pm1.
\end{equation}
Moreover,
\begin{equation}\label{minmonotono}\begin{split}
\mathcal{E}(u)\,&=\min \Big\{ \mathcal{E}(v) \;{\mbox{ s.t. }}\, v\in L^\infty(\R), \text{ monotone nondecreasing, s.t.  \eqref{LIMITS} holds}  \Big\}\\&\leq  \min\left\{\frac{4}{r}, 4+c_W\right\},\end{split}\end{equation}
where 
\[\mathcal{E}(v):=\frac1{2r^2}\int_\R \left(\osc{(x-r,x+r)} v\right)^2\,dx+
\int_{\R} W(v(x))\,dx.\]
\end{theorem}

We stress that the existence
of heteroclinic connections in nonlocal problems
is usually a rather difficult task in itself, which cannot
be achieved by standard ordinary differential equations
methods and cannot rely directly on conservation of energy formulas.
For problems modeled on fractional equations a careful
investigation of heteroclinic solutions and of their basic properties
has been recently performed in~\cite{A2, psv, a, cs, MR3460227, MR3594365, JAaasdfgN, cmy}.

The case that we treat in this paper is very different from
the existing literature, due to the lack of scale invariance.
In particular, the proof of the
existence result in Theorem~\ref{exthm} is more involved than the standard 
argument  based on direct methods,  due to the lack of appropriate compactness results for the oscillation functional. 
In order to gain compactness in our case,
we will need to prove that it is
possible to restrict the space of competitors for the Dirichlet problem
to monotone functions, and then we utilize suitable approximation
arguments in compact intervals.  

The technical arguments utilized in the proofs are
specifically tailored to our case, since we do not have any a priori information
on the regularity of the competitors involved in the minimization,
and solutions may be discontinuous. Therefore all the methods
based on pointwise analysis and geometric considerations
are not available at once in our setting, and they need to be replaced
by ad-hoc arguments.

See also~\cite{cdnv} for further discussions
about compactness issues for oscillatory functionals, and~\cite{cdnv2}
for existence and rigidity results for minimizers when~$W\equiv 0$. 
The case~$W\not\equiv 0$ that we consider in this paper cannot
be reduced to the existing literature on the subject, since it is the presence
of a nontrivial potential that defines the notion of equilibria
and makes the construction of 
heteroclinic orbits meaningful.

\subsection{Geometric properties   of minimal heteroclinic connections} \label{MAINRE2}

Now we describe the main characteristics of the
minimal heteroclinic connections given by
Theorem~\ref{exthm}.  In particular we will show that:
\begin{itemize}
\item they are monotone, \item they
satisfy an appropriate ``finite difference'' Euler-Lagrange equation, \item
and they are not necessarily continuous, since  there exists at least  one minimal heteroclinic connection which is piecewise constant on intervals of length $2r$. \end{itemize}
In our setting, the simplest competitor for heteroclinic connections is the
piecewise constant function defined (up to translations) as 
\begin{equation}\label{constant} 
u_0(x):=\begin{cases} 1&{\mbox{ if }} x>0,\\ -1 & {\mbox{ if }}x<0.\end{cases}
\end{equation}
It is easy to check that
\begin{equation}\label{constant2}
\mathcal{E}(u_0)=\frac{4}{r}.\end{equation}
A natural question is whether this is also a minimal heteroclinic connection. 
This would be the case if one considers
discontinuous double well potentials of the form~$W(t):=\chi_{(-1,1)}(t)$.
On the other hand, we can rule out the possibility that~$u_0$
is a minimizer when either~$r$ is sufficiently small or~$W$ is
sufficiently regular, as stated precisely in the next result: 

\begin{proposition}\label{propminimi} Assume that~\eqref{DOUBLEW} holds true. 
Then, the function $u_0$
is not a minimal heteroclinic connection if
\begin{itemize}
 \item either~$r\in\left(0,\frac{4}{4+c_W}\right)$,
being~$c_W$ defined in \eqref{cw},
\item or~$W$ is differentiable at $\pm 1$. 
\end{itemize}
\end{proposition} 
 
We also show that all the heteroclinic connections are necessarily monotone:
  
\begin{theorem}[Monotonicity of the heteroclinics]
\label{MONOTH} Assume that~\eqref{DOUBLEW} holds true.
Let $u\in L^\infty_{{\rm{loc}}}(\R)$ be a local minimizer to~\eqref{CALE}  
which satisfies \eqref{LIMITS}. 
Then,  $u$ is monotone nondecreasing.
\end{theorem} 

An interesting consequence of Theorem~\ref{MONOTH}
is that every minimal  heteroclinic connection satisfies an appropriate 
finite difference equation, which can be seen as an Euler-Lagrange equation
associated to the energy functional in~\eqref{CALE}. The precise result
that we have goes as follows:

\begin{theorem}[Euler-Lagrange equation]\label{euler}  Assume that~\eqref{DOUBLEW} holds true and
that $W$ restricted to $[-1,1]$ is a $C^1$ function. 
Let  $u\in L^\infty_{\rm loc}(\R)$ be a  a local minimizer for the functional
in~\eqref{CALE} which satisfies \eqref{LIMITS}. 

Then $u$ satisfies the Euler-Lagrange equation 
\begin{equation}\label{el1}
\frac{u(x+2r)+u(x-2r)-2u(x)}{r^2}=  W'(u(x))\qquad \mbox{ for a.e. $x\in\R$.}\end{equation}
\end{theorem}
 
We stress that it is not evident to obtain pointwise equations as in~\eqref{el1}
directly from the minimization of oscillation functionals as
in~\eqref{CALE} since, roughly speaking, it is not easy to carry the
derivatives inside the oscillation terms
(for instance, while in the classical case
one can obtain equation~\eqref{AKSMSJsdfSJ}
by simply taking derivatives of the functional in~\eqref{action},
this approach does not lead to equation~\eqref{el1}
by direct differentiation of the functional in~\eqref{CALE}). 

On the other hand,
it is always desirable to find necessary conditions
for minimization, and, in our case, the identity in~\eqref{el1}
plays an important role since it allows us to reconstruct certain values
of the minimizers by a partial knowledge of the values nearby.
In this sense, the operator on the left hand side of~\eqref{el1}
is a discretization of the second derivative, and~\eqref{el1}
can be seen as a discrete version of the classical pendulum and Allen-Cahn
equations
(compare with~\eqref{AKSMSJsdfSJ}, \eqref{AKSMSJsdfSJ2a}
and~\eqref{AKSMSJsdfSJ2}).\medskip

It is also interesting to
observe that, as~$r\searrow0$, the heteroclinic orbits found in Theorem~\ref{exthm}
recover the classical heteroclinics. This result makes use
of the Euler-Lagrange equation given by Theorem~\ref{euler},
and its statement goes as follows:

\begin{proposition}[Limit behavior as~$r\searrow0$]\label{ASY}
Assume that~\eqref{DOUBLEW} holds true and
that $W$ restricted to $[-1,1]$ is a $C^1$ function. 
For every~$r>0$, let~$u_r\in L^\infty(\R)$ be
a local minimizer of~\eqref{CALE}
which is monotone nondecreasing and satisfies~\eqref{LIMITS},
as given in Theorem~\ref{exthm}.

Then, up to a translation,
%and up to a subsequence
we have that~$u_r$ converges a.e.
to the  classical heteroclinic~$u$, namely  the unique solution to \begin{eqnarray}
\label{EQ:44} && 4u''(x)=W'(u(x))\qquad{\mbox{for all }}x\in\R,\\
\label{EQ:45}
&& u(0)=0,\\
\label{EQ:46} {\mbox{and }}&&\lim_{x\to\pm\infty}u(x)=\pm1.
\end{eqnarray}
\end{proposition}

In the next result, we show that minimal heteroclinic connections are not necessarily
regular, and this is a fundamental difference with respect to the classical
case of ordinary differential equations. To this end, we establish
the existence of a piecewise constant (and, in particular,
discontinuous) minimal heteroclinic connection. 

\begin{theorem}[Lack of regularity
and discontinuity of heteroclinc connections]  \label{costante}
Assume that~\eqref{DOUBLEW} holds true and moreover that $W$ restricted to $[-1,1]$ is
a $C^1$ function. 
Then,
there exists at least one  minimal heteroclinic connection $u$  
which is piecewise constant, on intervals of length $2r$. 

In particular,
there exists a sequence $(u_n)_{n\in \Z}$   such that 
\[u(x)\equiv u_n \qquad {\mbox{ for all }}x\in [2nr, 2(n+1)r).\]
The sequence $u_n$ is monotone nondecreasing, 
that is $u_n\leq u_{n+1}$, it satisfies $$
\lim_{n\to \pm \infty} u_n=\pm 1,$$ and  the recurrence relation \begin{equation}\label{ricorsiva1}
u_{n+2} = 2u_{n+1}-u_n +r^2 W'(u_{n+1}).
\end{equation}
\end{theorem}
 
We point out that the recurrence
relation \eqref{ricorsiva1} is the discrete version of the Euler-Lagrange
equation in~\eqref{el1}.

\subsection{Uniqueness issues}\label{IAKNSHDKNDN}
An interesting problem which is left open in  the previous
description of the minimal heteroclinic connections is the issue of uniqueness. 
In this direction, in the forthcoming Corollary~\ref{coro1}
we provide
a result about nonuniqueness of  monotone solutions
to the Euler-Lagrange equation \eqref{el1} which satisfy \eqref{LIMITS}.
This is based on the construction of two different heteroclinic
sequences satisfying the recurrence relation in~\eqref{ricorsiva1},
as given in the following result: 

\begin{proposition} \label{corodis}
Assume that~\eqref{DOUBLEW} holds true and that~$W$ restricted to $[-1,1]$ is a
$C^1$ function.
Then, there exist two different sequences~$(\bar w_n)_{n\in\Z}$ and $(\bar z_n)_{n\in\Z}$ which 
satisfy the following properties:
\begin{itemize}
\item $(\bar w_n)_{n\in\Z}$ and $(\bar z_n)_{n\in\Z}$
are  monotone
nondecreasing in~$n$, 
\item $(\bar w_n)_{n\in\Z}$ and $(\bar z_n)_{n\in\Z}$
satisfy the recurrence relation \eqref{ricorsiva1}, 
\item $(\bar w_n)_{n\in\Z}$ and $(\bar z_n)_{n\in\Z}$ satisfy
the limit property
\begin{equation}\label{1.13bis}
\lim_{n\to \pm\infty}\bar w_n=\lim_{n\to \pm\infty}\bar z_n=\pm1,\end{equation}
\item
$(\bar w_n)_{n\in\Z}$ and $(\bar z_n)_{n\in\Z}$
are odd sequences, in the sense that 
\begin{equation}\label{PARDI}
\begin{split}
& \bar w_0=0 \qquad{ \mbox{ and } }\qquad\bar w_n=-\bar w_{-n}
\qquad{ \mbox{ for all } }
n>0,\\ & \bar z_{n}= -\bar z_{-n-1}
\qquad{ \mbox{ for all } }
n\geq 0.\end{split}\end{equation}\end{itemize} 
\end{proposition} 

We recall that existence of heteroclinic solutions to discrete
recurrence relations such as \eqref{ricorsiva1} has been
also considered in the literature, see e.g.~\cite{xiao1,  hxt}.
A consequence of Proposition \ref{corodis} is the following result.

 \begin{corollary}[Lack of uniqueness for heteroclinic solutions]\label{coro1}
 Assume that~\eqref{DOUBLEW} holds true and that~$W$ restricted to $[-1,1]$ is a
$C^1$ function.
Then, there exist two geometrically
 different monotone nondecreasing functions $u,v:\R\to [-1,1]$ which satisfy~\eqref{LIMITS} and are solutions to the
 Euler-Lagrange equation~\eqref{el1}. \end{corollary} 
 
\subsection{The Dirichlet problem}\label{THEDIRIC}
We now observe that the lack of regularity that we pointed out for solutions to the Euler-Lagrange equation \eqref{el1}
is a general phenomenon  in  equations involving the  discrete difference operator
\begin{equation}\label{Dr} D_r u(x):= \frac{u(x+r)+u(x-r)-2u(x)}{r^2}.\end{equation}
In particular, we consider  the Dirichlet problem associated to this operator with source term $f\in L^{\infty}_{{\rm{loc}}}(\R)$ and boundary data 
$\alpha, \beta\in L^{\infty}_{{\rm{loc}}}(\R)$:
\begin{equation}\label{DS} \left\{\begin{matrix}
D_r u=f&{\mbox{ in }}(a,b),\\
u=\alpha&{\mbox{ in }}[a-r,a],\\
u=\beta&{\mbox{ in }}[b,b+r].
\end{matrix}\right.\end{equation}
In this setting, we provide basic existence, uniqueness and regularity properties of
the solutions of~\eqref{DS}.
We start by showing that there exists a unique solution to \eqref{DS}, according to the following result: 
\begin{theorem}[Dirichlet problem for~$D_r$]\label{DIRI}
The system in~\eqref{DS} has a unique solution $u$ (up to sets of zero measure), which is given by the function
\begin{equation}\label{SOLUZIONE} u(x)
:= \left\{
\begin{matrix}
\alpha(x) & {\mbox{ if }} x\in[a-r,a],\\
\\
\\
\begin{matrix}
\displaystyle\frac{\overline k(x)}{\overline k(x)+\underline k(x)}\,\Bigg(\alpha(x-\underline k (x)r)-r^2\, \displaystyle\sum_{j=1}^{\underline k(x)-1} j f\big(x-(\underline k(x)-j)r\big)\Bigg)
 \\ 
+ \displaystyle\frac{\underline k(x)}{\overline k(x)+\underline k(x)}\,\Bigg(\beta(x+\underline k (x)r)-r^2  \,\displaystyle\sum_{j=1}^{\overline k(x)-1} j f\big(x+(\overline k(x)-j)r\big) \Bigg)
\\
 -r^2\displaystyle\frac{\underline k(x)\overline k(x)}{\overline k(x)+\underline k(x)} \,f(x)  
\end{matrix} & {\mbox{ if }}x\in(a,b),\\
\\
\\
\beta(x) & {\mbox{ if }}x\in[b,b+r],
\end{matrix}\right.
\end{equation}
where
\begin{equation}\label{DK}
\overline k(x):=\left\lceil \frac{b-x}{r}\right\rceil \qquad{\mbox{and}}\qquad
\underline k(x):=\left\lceil \frac{x-a}{r}\right\rceil.
\end{equation}
\end{theorem}

A useful tool towards the proof of the uniqueness result in Theorem~\ref{DIRI}
consists in a suitable
Maximum Principle for the operator $D_r$ (which will be presented in Lemma~\ref{MAX PLE}).
\medskip

We analyze now the regularity of the solution to \eqref{DS}. 
In particular, we obtain a uniform bound on the solution in terms of
the external data~$\alpha$ and~$\beta$, and of the source function~$f$.
Then, we bound the possible jumps of the solution by a quantity that depends
on~$r$, $\alpha$, $\beta$ and~$f$ (and which becomes small as~$r\searrow0$).

\begin{corollary}[Continuity and jump bounds for the Dirichlet problem for~$D_r$]\label{COROCONT}
Let ~$\alpha\in C([a-r,a])$, $\beta\in C([b,b+r])$, and~$f\in C([a,b])$,
and let~$u$ be the solution
to~\eqref{DS}.

Then, $u\in L^\infty([a-r,b+r])$, with
\begin{equation}\label{BOUND LINF}
\| u\|_{ L^\infty([a-r,b+r]) }\le \|\alpha\|_{ L^\infty([a-r,a]) }
+\sup_{{p\in[a-r,a]}\atop{q\in[b,b+r]}}|\alpha(p)-\beta(q)|+\big((b-a)^2+r^2\big)\,\|f\|_{L^\infty([a,b])}.
\end{equation} 

Also, we have that~$u\in
C([a-r,b+r]\setminus {\mathcal{J}})$,
where
$$ {\mathcal{J}}:= (a+r\N)\cup(b-r\N).$$
Moreover, at any points of~${\mathcal{J}}$,
the function~$u$ jumps by at most
\begin{equation}\label{JUMPATMOST}
\sup_{[a-r,a]}\alpha-\inf_{[a-r,a]}\alpha +
\frac{r}{b-a}\sup_{{p\in[a-r,a]}\atop{q\in[b,b+r]}}|\alpha(p)-\beta(q)|+
\frac{Cr}{b-a}\,\big( (b-a)^2+r^2\big)\,\|f\|_{L^\infty([a,b])}
\end{equation}
for some $C>0$ depending on $a,b,r$. 
\end{corollary}

It is interesting to observe that the jump bound in~\eqref{JUMPATMOST}
improves as~$r\searrow0$ and in fact it recovers continuity in the limit
(and this fact can be also compared with
the asymptotic result of Proposition~\ref{ASY}).
As a counterpart of this observation, we stress that
the Dirichlet problem run by the operator~$D_r$ does exhibit,
in general, discontinuous solutions:

\begin{corollary}
[Lack of regularity
and discontinuity of the solutions of the Dirichlet problem]  \label{costanteDIR}
Fix~$n\in\N$, with~$n>1$.
The solution of the
Dirichlet problem \eqref{DS}
with~$r:=1/n$, $\alpha:=0$, $\beta:= 1$, $f:= 0$, $a:=0$ and~$b:=1$
is discontinuous.
\end{corollary}

The discontinuous example in Corollary~\ref{costanteDIR}
can be seen as a natural counterpart in the setting of the Dirichlet problem~\eqref{DS}
of the phenomenon discussed
in Theorem~\ref{costante} in the case of global heteroclinics.

\subsection*{Plan of the paper} 
The rest of this paper is organized as follows. 
Section \ref{secex} contains the construction of a local minimizer
to~\eqref{CALE} which connects monotonically $-1$ and $1$,
that is the proof of Theorem~\ref{exthm}.
This construction is obtained by approximation,
by solving suitable Dirichlet problems.
In Section \ref{secmon} we prove that every local minimizer to~\eqref{CALE}
which connects the two variationally
stable equilibria is monotone, namely Theorem~\ref{MONOTH}.
 Section \ref{seceuler} is devoted to the proof of Theorem \ref{euler}. This is obtained by introducing a new functional $\mathcal{F}$,
 which coincides with $\mathcal{E}$ on monotone functions.
 
The asymptotics as~$r\searrow0$ is discussed in Section~\ref{ASYS},
which contains the proof of Proposition~\ref{ASY}.
Then, Section \ref{secdiscrete} contains the proofs of Theorem \ref{costante}, Proposition \ref{corodis}, and Corollary \ref{coro1},  so in particular it contains the analysis of the discrete version of the 
 Euler-Lagrange equation \eqref{el1}, and the nonuniqueness issues described
in Proposition \ref{corodis}
and Corollary \ref{coro1} are discussed in Section~\ref{NOAM:S}.

Finally, in Section~\ref{secdir} we consider the
Dirichlet problem for $D_r$, and we present the proofs of
Theorem~\ref{DIRI}, and of Corollaries \ref{COROCONT}
and~\ref{costanteDIR}.

\subsection*{Notation}  
In the $\sup$, $\inf$, $\limsup$ and $\liminf$ notation, we mean
the ``essential supremum and infimum'' of the function (i.e.,
sets of null measure are neglected) and the essential   superior and  inferior limit of a function at a point.
Moreover, we shall identify a set~$E\subseteq \R^n$ with its points of
density one, and~$\partial E$ with the topological boundary of the set of points of density one. 

For $x\in \R$, we will denote   with $\left\lceil x\right\rceil $  (resp. $\left\lfloor x\right\rfloor $ )
the smallest integer~$z$ such that~$x\le z$ (the biggest   integer~$z$ such that~$x\ge z$) that is 
$$\left\lceil x\right\rceil:=\min \{z\in\Z {\mbox{ s.t. }} x\leq z\}\qquad(\text{resp. }\left\lfloor x\right\rfloor:=\max\{z\in\Z {\mbox{ s.t. }} x\geq z\}).$$
Finally for any  $u:I\subset \R \to \R$ monotone function,
we will always identify $u$ with its   right continuous representative.

\section{Existence of minimal heteroclinic connections, and proofs of
Theorem~\ref{exthm} and Proposition~\ref{propminimi}}\label{secex}

The construction of the local minimizer given by Theorem~\ref{exthm}
will be obtained by approximations,
using solutions to appropriate Dirichlet problems. 

To this end, we fix $R>2r$ and consider the minimization problem
\begin{equation}\label{dirpb}
\inf \Big\{ \mathcal{E}_{(-R,R)} (u) {\mbox{ s.t. }} u\in L^{\infty}(\R),
\, u(x)=1 \,\mbox{ for a.e. $x\geq R-r$ and }\,u(x)=-1\,\mbox{ for a.e. $x\leq -R+r$}\Big\}. 
\end{equation} 
To prove that \eqref{dirpb} admits a minimum, we will use standard direct
method in the calculus of variations. First of all, though, we need to 
restrict the space of competitors to gain some more compactness. Namely, we prove that
we can consider monotone nondecreasing competitors, as stated in the following result:

\begin{lemma}\label{competit} 
Assume that~\eqref{DOUBLEW} holds true.
Let $R>2r$ and~$v\in L^{\infty}(\R)$, with~$ v(x)=1$ for a.e. $x\geq R-r$
and~$v(x)=-1$ for a.e. $x\leq -R+r$. 

Then, there exists a monotone nondecreasing function $\tilde v$
such that $\tilde v=v$ in $(-\infty, -R+r)\cup (R-r, +\infty)$
and  $$\mathcal{E}_{(-R,R)} (\tilde v)\leq \mathcal{E}_{(-R,R)} (v).$$
\end{lemma} 

\begin{proof} 
We first prove that it is enough to consider competitors~$v$ taking
values in~$[-1,1]$. For this, we claim that
\begin{equation}\label{comp11}
\mathcal{E}_{(-R,R)}\big(\max\{-1, \min\{1, v\}\}\big)\leq \mathcal{E}_{(-R,R)}(v).
\end{equation}
To prove~\eqref{comp11}, we observe that, by definition, for every $c\in\R$ and  $x\in\R$,
\begin{equation}\label{taglio}
\osc{(x-r,x+r)} u= \osc{(x-r,x+r)} \min \{u,c\}+\osc{(x-r,x+r)} \max\{u,c\}.\end{equation} 
Then, given~$v$ as in the statement
of Lemma~\ref{competit}, by \eqref{taglio} we get that for all $x\in\R$ 
\begin{equation}\label{irgbfbvbv58u6}
\osc{(x-r,x+r)} \max\{-1, \min\{1, v\}\}\leq \osc{(x-r,x+r)} v.\end{equation}
Moreover, by the hypothesis on~$W$ in \eqref{DOUBLEW}, we have that
\begin{equation}\label{irgbfbvbv58u62}\int_{-R}^{R} W\big(
\max\{-1, \min\{1, v\}\}(x)\big)\,dx
\leq \int_{-R}^{R} W(v(x))\,dx.\end{equation}
Hence, \eqref{comp11} follows from~\eqref{CALE},
\eqref{irgbfbvbv58u6}, and~\eqref{irgbfbvbv58u62}.

Now, if~$v$ is monotone nondecreasing, the proof of Lemma~\ref{competit}
is completed by taking~$\tilde v:=v$.

Hence, we suppose that $v$ is
not monotone nondecreasing, and we provide a method to
modify~$v$ in~$[-R+r,R-r]$
in order to get a monotone nondecreasing function~$\tilde v$ with lower energy, as desired.

Since $v$ is not monotone nondecreasing,
there exist $a$, $b\in\R$ such that
\begin{equation}\label{aoboaobo2}
a<b,\qquad 
B:=\liminf_{x\to b}v(x)<\limsup_{x\to a}v(x)=:A
\quad \mbox{ and  }\quad
B\leq v(x)\leq A \quad
\mbox{ for a.e. $x\in [a,b]$}.\end{equation}
The idea is to consider all possible quadruples as in \eqref{aoboaobo2},
by substituting $v$ with
a function $\tilde v$ which coincides with~$v$ outside~$[-R+r, R-r]$ and 
has lower energy than $v$, and this will imply the thesis of Lemma~\ref{competit}.
The precise details go as follows.

If $A=1$, we define 
\begin{equation}\label{dprhrnhyi}
\tilde v(x):= \begin{cases} 
1=\max\{v(x), 1\} &{\mbox{ if }} a\leq x\leq R-r,\\
v(x) &{\mbox{ otherwise}}.
\end{cases} 
\end{equation}
Using \eqref{taglio} and the
first assumption on~$W$ in \eqref{DOUBLEW},
we conclude that $$\mathcal{E}_{(-R,R)}(\tilde v)\leq \mathcal{E}_{(-R,R)}(v).$$
Similarly, if $B=-1$, we define
\begin{equation}\label{ir7gfdh}
\tilde v(x):= \begin{cases} -1=\min\{v(x),-1\} &\quad \mbox{ if } -R+r\leq x< b,\\ 
v(x)  & \quad \mbox{ otherwise},\end{cases} \end{equation}
and we get that
$$\mathcal{E}_{(-R,R)}(\tilde v)\leq \mathcal{E}_{(-R,R)}(v).$$ 

As a consequence, from now on we assume that~$-1<B<A<1$.
We define 
\[B_0:=\inf_{x\geq a} \liminf_{y\to x} v(y).\]
Observe that  $B_0\leq B$ and $v(x)\geq B_0$ for all $x\geq a$.

If $B_0=B$, we set $b_0=b$. If $B_0<B$, then 
let $\eta_j$ such that $v(\eta_j)\to B_0$.
Then up to extracting a subsequence, we get that $\eta_j\to b_0$, for some~$b_0\in\R$.
In this case, due to~\eqref{aoboaobo2}, we have that~$b_0>b>a$.

We also notice that~$B_0>-1$, otherwise, if~$B_0=-1$, we argue as before,
defining~$\tilde v$ as in~\eqref{ir7gfdh} with~$b_0$ in place of~$b$, and obtaining that~$
\mathcal{E}_{(-R,R)}(\tilde v)\leq \mathcal{E}_{(-R,R)}(v)$. 

Now, if $A+B_0\geq 0$, then since $A>B_0$, by assumption \eqref{DOUBLEW}, we get that $W(t)\geq W(A)$ for all $t\in [B_0,A]$. We define
\[\tilde v(x)= \begin{cases}
\max\{v(x), A\} &{\mbox{ if }} a\leq x\leq R-r,\\
v(x) & {\mbox{ otherwise}},\\ 
\end{cases} 
\] and, recalling that $v(x)\geq B_0$ for all $x\geq a$, we conclude
that~$\mathcal{E}_{(-R,R)}(\tilde v)\leq \mathcal{E}_{(-R,R)}(v)$. 
 
Hence, we suppose that~$A+B_0<0$, and we define 
\[A_0:=\sup_{x\leq b_0} \limsup_{y\to x} v(y).\]
Observe that $A_0\geq A$, and $v(x)\leq A_0$ for all $x\leq b_0$.  
 
If $A_0=A$, since $A+B_0<0$, we have that $W(t)\geq W(B_0)$ for all $t\in [B_0,A]$.
So, we define 
\[\tilde v(x)= \begin{cases} \min\{v(x), B_0\} &{\mbox{ if }} -R+r\leq x<b_0,\\ 
v(x) &{\mbox{ otherwise}},\end{cases} 
\] and, recalling that $v(x)\leq A_0=A$ for all $x\leq b_0$, we conclude
that~$\mathcal{E}_{(-R,R)}(\tilde v)\leq \mathcal{E}_{(-R,R)}(v)$. 
 
Assume now that $A_0>A$. Let $\eta_j$ such that $v(\eta_j)\to A_0$.
Then up to extracting a subsequence, we get that $\eta_j\to a_0$, for some~$a_0<b_0$. 
 
We observe that if $A_0=1$, we argue as before, defining~$\tilde v= 1$ 
as in~\eqref{dprhrnhyi} with~$a_0$ in place of~$a$,
and we conclude that~$\mathcal{E}_{(-R,R)}(\tilde v)\leq \mathcal{E}_{(-R,R)}(v)$. 

Now we iterate this procedure, setting
\[B_1:=\inf_{x\geq a_0} \liminf_{y\to x} v(y),\]
and noticing that  $B_1\leq B_0$ and $v(x)\geq B_1$ for all $x\geq a_0$. 

If $B_0=B_1$, we consider two cases: either~$A_0+B_0\leq 0$ or~$A_0+B_0> 0$.
If $A_0+B_0\leq 0$, we define~$\tilde v$ as in~\eqref{ir7gfdh} with~$b_0$ in place of~$b$,
and we conclude again that~$\mathcal{E}_{(-R,R)}(\tilde v)\leq \mathcal{E}_{(-R,R)}(v)$.
If instead~$A_0+B_0> 0$, we set~$\tilde v$ as in~\eqref{dprhrnhyi} with~$a_0$ in place of~$a$,
obtaining that~$\mathcal{E}_{(-R,R)}(\tilde v)\leq \mathcal{E}_{(-R,R)}(v)$. 
 
So, we are left with the case $B_1<B_0$. The possibility that~$B_1=-1$ can be 
dealt with as before.
Hence, if~$B_1>-1$, we define $b_1$ such that 
$$\liminf_{y\to b_1} v(y)=B_1.$$ 
Observe that necessarily $a_0<b_1<a<b_0$. So if $x\geq a_0$ we have that~$v(x)\geq B_1$
and if $x\leq b_1$ we have that~$v(x)\geq A_1$. 
As above, we separate two cases, namely we consider the case in which~$A_0+B_1\leq 0$
and the one in which~$A_0+B_1> 0$. In the first case,
we define~$\tilde v$ 
as in~\eqref{ir7gfdh} with~$b_1$ in place of~$b$, while in the second case
we set~$\tilde v$
as in~\eqref{dprhrnhyi} with~$a_0$ in place of~$a$. In both cases, we obtain that~$
\mathcal{E}_{(-R,R)}(\tilde v)\leq \mathcal{E}_{(-R,R)}(v)$.

These observations took into account all possible cases given by~\eqref{aoboaobo2},
and so the proof of Lemma~\ref{competit} is complete.
\end{proof}

With the aid of Lemma \ref{competit}, we can prove existence of a solution to the minimization
problem in~\eqref{dirpb}.

\begin{proposition}\label{exdir}
For every $R>2r+1$,  there exists $u_R\in L^\infty(\R)$ solution of the minimization problem
in~\eqref{dirpb}. Moreover, $u_R$ is monotone nondecreasing.

In addition,
\begin{equation}\label{ENBOUND}
\mathcal{E}_{(-R, R)}(u_R)\leq  \min\left\{\frac{4}{r}, 4+c_W\right\}
\end{equation}
where $c_W$ is as in \eqref{cw}. 
\end{proposition}

\begin{proof} In light of Lemma \ref{competit},
the minimization problem in~\eqref{dirpb} is equivalent to the
following minimization problem 
\begin{equation}\label{fy6458hg}
\inf_{v\in \mathcal{M}_R} \mathcal{E}_{(-R, R)}(v),
\end{equation}
where  
\begin{equation}\begin{split}\label{emmerre}
\mathcal{M}_R :=\,&\big\{v\in L^\infty(\R) \mbox{ s.t. $v$ is monotone 
nondecreasing, }\\
&\quad  v(x)=1 {\mbox{ for a.e. }} x\geq R-r {\mbox{ and }}
v(x)=-1\ {\mbox{ for a.e. }} x\leq -R+r\big\}.\end{split}\end{equation}

We start proving \eqref{ENBOUND}. We consider the function~$v_{\pm 1}\in \mathcal{M}_R$
with~$v_{\pm1}:=\pm1$
in~$(-R+r,R-r)$. Then, in view of the
properties of~$W$ given in~\eqref{DOUBLEW}, 
$$ \mathcal{E}_{(-R, R)}(v_{\pm1})= 
\frac{4}{r},$$
and this implies that
\begin{equation}\label{ecostante}
\inf_{v\in \mathcal{M}_R} \mathcal{E}_{(-R, R)}(v)\leq \frac{4}{r}.\end{equation}
Furthermore, we let ~$\tilde v \in  \mathcal{M}_R$ such that 
\begin{equation}\label{vtilde}  \tilde v(x):= \begin{cases} 
 x+R-r-1 & {\mbox{ for any }} x\in [-R+r, 2-R+r],\\
1 & {\mbox{ for any }} x\in [ 2-R+r, R-r].
\end{cases}.\end{equation}
We note that~$2-R+r<R-r$, since $R>1+r$.
By \eqref{DOUBLEW} and~\eqref{cw},
we get that
\begin{equation}\label{2.13bis}
\int_{-R}^{R} W(\tilde v(x))\,dx= \int_{-R}^{2-R+r} W(x+R-r-1)\,dx
= \int_{-1}^1 W(s)\, ds=c_W.\end{equation}
Since  $r\leq 1$    (and so in particular $-R+2r\leq -R+2$) and $R>r+1$, we get 
 \begin{eqnarray}\label{contotilda}
 &&\frac{1}{2r^2} \int_{-R}^R \Big(\osc{(x-r,x+r)} \tilde v\Big)^2\,dx
 \\ 
&= &\frac{1}{2r^2}  \int_{-R}^{-R+2r} \big(\tilde v(x+r)+1\big)^2\, dx
+  \frac{1}{2r^2} \int_{-R+2r}^{-R+2} \big(\tilde v(x+r)-\tilde v(x-r)\big)^2 \,dx \nonumber \\
&&\qquad +  \frac{1}{2r^2} \int_{-R+2}^{-R+2+2r} \big(1- \tilde v(x-r)\big)^2\, dx \nonumber \\
&= &  \frac{1}{2r^2}  \int_{0}^{2r} x^2 \,dx+
\frac{1}{2r^2} \,4r^2(2-2r)+  \frac{1}{2r^2} \int_{0}^{2r}x^2\, dx\nonumber \\
&=&\frac{8}{3} r+ 4(1-r).\nonumber\end{eqnarray}
As a consequence of this and~\eqref{2.13bis},
$$ {\mathcal{E}}_{(-R,R)}(\tilde v)= \frac{8}{3} r+ 4(1-r)+ c_W\leq 4+c_W.$$
{F}rom this and~\eqref{ecostante}, we obtain~\eqref{ENBOUND}. 

We show now that a minimizer~$u_R$ does exist.
To this aim, we consider a minimizing sequence~$u_n\in \mathcal{M}_R$, and we have that~$u_n$
is uniformly bounded. Moreover, the sequence~$u_n$ has uniformly bounded variation
(since it consists of
equibounded monotone functions). By 
compact embeddings of~$BV(-R+r,R-r)$ in~$L^p(-R+r, R-r)$ for every $p\geq 1$
(see \cite[Corollary 3.49]{MR1857292}), we obtain that, up to a subsequence,
$u_n\to u_R$
pointwise and in $L^1(-R+r,R-r)$, as~$n\to+\infty$, for some~$u_R\in\mathcal{M}_R$.

Now, by the lower semicontinuity of
the oscillation functional with respect to $L^1$ convergence (see~\cite{cdnv}),
and the continuity of the potential term of the energy with respect to pointwise
convergence,
we conclude that~$u_R$ is a solution to~\eqref{fy6458hg}, and therefore to~\eqref{dirpb}.
\end{proof} 

Now we are in the position of completing the proof of Theorem~\ref{exthm}.

\begin{proof}[Proof of Theorem \ref{exthm}]
For any~$R>2r+1$, 
we consider the solution $u_R$ of~\eqref{dirpb} constructed in Proposition \ref{exdir}.
We recall the notation in~\eqref{emmerre}, and we
observe that~$u_{R}\in \mathcal{M}_{R'}$, for any~$R'>R$.
Hence, we obtain that
\begin{equation}\label{mongiu} 
e_R:= \mathcal{E}_{(-R, R)}(u_R)=\mathcal{E}_{(-R', R')}(u_R)\geq \mathcal{E}_{(-R', R')}
(u_{R'})=:e_{R'}.\end{equation} 
Therefore, recalling also the uniform bound in~\eqref{ENBOUND},
we conclude that
\begin{equation}\label{limer}
\lim_{R\to+\infty} e_R=\inf_{R>2r} e_R=:e\in\left[ 0, \min\left\{\frac{4}{r},4+c_W\right\}\right]. 
\end{equation}

Now, up to a translation, we can assume that, for all~$R>2r+1$,
$u_R(x)<0$ for any~$x<0$ and $u_R(x)\geq 0$ for any~$x\geq 0$.
Moreover, 
we observe that the sequence~$u_R$ is equibounded,
and has equibounded total variation, since the functions~$u_R$ are all monotone.
Thus, by compactness theorem
(see \cite[Corollary 3.49]{MR1857292}) we get that, up to extracting a subsequence,
$u_R\to u$ pointwise and locally in $L^p$ for every $p\geq 1$, as~$R\to+\infty$. 
We point out that
\begin{equation}\label{SODDI}
{\mbox{the limit function~$u\in L^\infty(\R)$, with~$|u|\le1$,
it is monotone nondecreasing and satisfies~\eqref{LIMITS}.}} 
\end{equation}

Now, we set 
$$ \mathcal{E}(v):=\frac1{2r^2}\int_\R \left(\osc{(x-r,x+r)} v\right)^2\,dx+
\int_{\R} W(v(x))\,dx,$$
we recall~\eqref{limer} and we claim that 
\begin{equation}\label{cen}
e=\mathcal{E}(u).\end{equation}
For this, we fix $M>0$. Then,
by the lower semicontinuity of the oscillation part of the functional with respect
to $L^1$ convergence and 
the continuity of the potential part with respect to the pointwise convergence, we get 
\begin{eqnarray*}
0&\leq& {\mathcal{E}}_{(-M,M)} (u)\\
&=&
\frac1{2r^2}\int_{-M}^M \left(\osc{(x-r,x+r)} u\right)^2\, dx+\int_{-M}^{M} W(u(x))\,dx\\
&\leq& \liminf_{R\to +\infty} \left[
\frac1{2r^2}\int_{-M}^M \left(\osc{(x-r,x+r)} u_R\right)^2\, dx
+\int_{-M}^{M} W(u_R(x))\,dx\right]\\
&\leq  & \liminf_{R\to +\infty}\left[
\frac1{2r^2}\int_{\R} \left(\osc{(x-r,x+r)} u_R\right)^2+\,
dx+\int_{\R} W(u_R(x))\,dx\right]\\
&=& \lim_{R\to +\infty} e_R= e, \end{eqnarray*}
where we used the notation in~\eqref{mongiu}.
Consequently, since ${\mathcal{E}}_{(-M,M)} (u)$ is monotone nondecreasing in~$M$, 
sending~$M\to+\infty$, we conclude that
\begin{equation}\label{SODDI2}
0\leq \mathcal{E}(u) \leq e. \end{equation}

Now, for any~$v\in L^\infty(\R)$ such that~$v$ is monotone nondecreasing and
\begin{equation}\label{LIMITSV}
-1\leq v\leq 1,\qquad
\lim_{x\to\pm\infty}v(x)=\pm1,\quad 
{\mbox{and }}\quad \mathcal{E}(v)<+\infty.\end{equation}
and 
for any~$M>0$, we define the function
\begin{equation}\label{tagliov} v^M(x)= \begin{cases} v(x), &{\mbox{ if }}
x\in  (-M+r,M-r),\\ 1, & {\mbox{ if }} x\geq M-r,\\
-1, & {\mbox{ if }} x\leq -M+r.\end{cases} \end{equation}
We claim that
\begin{equation}\label{eruiher0579}
\lim_{M\to+\infty} \mathcal{E}(v^M)=\mathcal{E}(v).
\end{equation}
For this, we fix~$\eps>0$ and we take~$M$ sufficiently large such that
$$ |v(x)- 1|+|v(y)+1|\le \eps, \quad {\mbox{ for any }} x\in [M-3r,+\infty)
\quad {\mbox{ and for any }}y\in(-\infty,-M+3r],$$
in light of~\eqref{LIMITSV}. 
This gives that, for any~$x\in(-\infty,-M+2r]\cup [M-2r,+\infty)$,
$$ \osc{(x-r,x+r)} v\le 2\eps,$$
and so, for any~$x\in(-\infty,-M+2r]\cup [M-2r,+\infty)$,
\begin{equation}\label{06y6jjh}
\osc{(x-r,x+r)} v^M\le 2\eps.\end{equation}
Also,
\begin{eqnarray*}
&& \frac1{2r^2}\int_\R \left(\osc{(x-r,x+r)} v^M\right)^2\,dx\\
&=&
\frac1{2r^2}\int_{-M+2r}^{M-2r} \left(\osc{(x-r,x+r)} v^M\right)^2\,dx+
\frac1{2r^2}\int_{M-2r}^{+\infty} \left(\osc{(x-r,x+r)} v^M\right)^2\,dx
+\frac1{2r^2}\int_{-\infty}^{-M+2r} \left(\osc{(x-r,x+r)} v^M\right)^2\,dx\\
&=&
\frac1{2r^2}\int_{-M+2r}^{M-2r} \left(\osc{(x-r,x+r)} v\right)^2\,dx+
\frac1{2r^2}\int_{M-2r}^{M} \left(\osc{(x-r,x+r)} v^M\right)^2\,dx
+\frac1{2r^2}\int_{-M}^{-M+2r} \left(\osc{(x-r,x+r)} v^M\right)^2\,dx.
\end{eqnarray*}
Therefore, using~\eqref{06y6jjh},
\begin{equation*}\begin{split}
& \left|\frac1{2r^2}\int_\R \left(\osc{(x-r,x+r)} v^M\right)^2\,dx
-\frac1{2r^2}\int_\R \left(\osc{(x-r,x+r)} v\right)^2\,dx\right|\\
\le\;&
\left|\frac1{2r^2}\int_{-M+2r}^{M-2r} \left(\osc{(x-r,x+r)} v\right)^2\,dx
-\frac1{2r^2}\int_\R \left(\osc{(x-r,x+r)} v\right)^2\,dx\right|+
\frac{8\epsilon^2}{r}.
\end{split}\end{equation*}
As a consequence,
\begin{eqnarray*}
&& \left|\mathcal{E}(v^M)-\mathcal{E}(v)\right|\\
&\le& 
\left|\frac1{2r^2}\int_{-M+r}^{M-r} \left(\osc{(x-r,x+r)} v\right)^2\,dx
-\frac1{2r^2}\int_\R \left(\osc{(x-r,x+r)} v\right)^2\,dx\right|+
\frac{8\epsilon^2}{r}\\
&&\qquad + \left|\int_{\R} W(v^M(x))\,dx -
\int_{\R} W(v(x))\,dx\right|\\
&\le& 
\left|\frac1{2r^2}\int_{-M+2r}^{M-2r} \left(\osc{(x-r,x+r)} v\right)^2\,dx
-\frac1{2r^2}\int_\R \left(\osc{(x-r,x+r)} v\right)^2\,dx\right|+
\frac{8\epsilon^2}{r}\\
&&\qquad +\left|\int_{-M}^{M} W(v^M(x))\,dx -
\int_{\R} W(v(x))\,dx\right|,
\end{eqnarray*}
which implies the desired result in~\eqref{eruiher0579}
sending~$\eps\to0$ and~$M\to+\infty$.

{F}rom~\eqref{eruiher0579}, we have that
for any~$\eps>0$ there exists~$M(\eps, v)$ such that 
\begin{equation}\label{stimavM} \mathcal{E}(v)\geq \mathcal{E}(v^M)-\eps \qquad {\mbox{ for all }}
M\geq M(\eps, v).\end{equation}
We also observe that $v^M\in \mathcal{M}_M$, and therefore, by the minimality of $u_M$
and using~\eqref{stimavM},
\eqref{mongiu} and~\eqref{limer},
we obtain that
\[\mathcal{E}(v)\geq \mathcal{E}(v^M)-\eps\geq \mathcal{E}_{(-M,M)}(v^M)-\eps\geq  \mathcal{E}_{(-M,M)}(u_M)-\eps=e_M-\eps\geq e-\eps.\]  
By the arbitrariness of $\eps$, we conclude that~\begin{equation}\label{vm} \mathcal{E}(v)\geq e.\end{equation}

Now we observe that~$u$ satisfies~\eqref{LIMITSV}, in view of~\eqref{SODDI} and~\eqref{SODDI2},
and so we can take~$v:=u$, obtaining that~$\mathcal{E}(u)\geq e$.
This and~\eqref{SODDI2} give
the claim in~\eqref{cen}. 
Moreover, \eqref{vm} and \eqref{cen} imply directly \eqref{minmonotono}. 

In order to complete the proof of Theorem~\ref{exthm}, it only remains to show that
\begin{equation}\label{LOCAM}
{\mbox{$u$ is a local minimizer for the functional in~\eqref{CALE}.}}\end{equation}
To this end, we argue towards a contradiction, assuming that there exist~$M_0>0$,
a function~$v\in L^{\infty}(\R)$ such that
\begin{equation}\label{96u7jgnbvfbfe}
{\mbox{$v=u$ outside~$[-M_0+r,M_0-r]$,}}\end{equation} and~$\eps>0$
such that
\begin{equation}\label{contr1} \mathcal{E}_{(-M_0, M_0)}(v)\leq \mathcal{E}_{(-M_0, M_0)}(u)-
2\eps.\end{equation}
Also, recalling \eqref{limer} and~\eqref{cen}, and using~\eqref{stimavM}
(with~$v:=u$ and~$v^M:=u^M$ defined in \eqref{tagliov}), we get that there exists  $M_1>2r$ such that for all $M\geq M_1$,  
\begin{equation}\label{contr2}
e_M\geq e=\mathcal{E}(u)\geq \mathcal{E}(u^M)-\eps.\end{equation}
Now we take~$M>\max\{M_1,M_0+r\}$ and we consider~$v^{M}$ as given in \eqref{tagliov}.
Then, we get from~\eqref{96u7jgnbvfbfe} and~\eqref{contr1} that
\[\mathcal{E}_{(-M, M)}(v^{M})\leq \mathcal{E}_{(-M, M)}(u^{M})-2\eps.\] 
Consequently, recalling also~\eqref{contr2} and the notation in~\eqref{mongiu}, we conclude that
\begin{equation}\label{irhbjf7734gfybf}
\mathcal{E}_{(-M, M)}(v^{M})\leq \mathcal{E}_{(-M, M)}(u^{M})-
2\eps\leq\mathcal{E}(u^M)-2\eps \leq  \mathcal{E}(u)-\eps\leq e_M-\eps= 
\mathcal{E}_{(-M,M)}(u_M)-\eps. \end{equation}
Now from Lemma~\ref{competit} we know that there exists a monotone
nondecreasing function~$\tilde{v}^M$ such that~$\tilde{v}^M=v^M$
in~$(-\infty, -M+r)\cup(M-r,+\infty)$ such that
$$ {\mathcal{E}}_{(-M,M)}(\tilde{v}^M)\le 
{\mathcal{E}}_{(-M,M)}({v}^M).$$
Furthermore, the function~$\tilde{v}^M$ belongs to~$\mathcal{M}_M$
(recall the definition of this space in~\eqref{emmerre}).
As a consequence of this observation and of~\eqref{irhbjf7734gfybf}, we find that
$${\mathcal{E}}_{(-M,M)}(\tilde{v}^M)\le 
\mathcal{E}_{(-M,M)}(u_M)-\eps,$$
which is in contradiction with the minimality of $u_M$.
This conclude the proof of Theorem~\ref{exthm}.
\end{proof} 

We conclude the section proving Proposition \ref{propminimi}.

\begin{proof}[Proof of Proposition \ref{propminimi}]
In light of~\eqref{constant2},
to prove the statement,
it is sufficient to construct
a function~$v:\R\to \R$, which is monotone nondecreasing,
satisfies~\eqref{LIMITS} and such that
$$\mathcal{E}(v)<\frac{4}{r}=\mathcal{E}(u_0).$$ 
To this end, we observe that, if $r<\frac{4}{4+c_W}$, then
$$ 4+c_W<\frac{4}{r},$$
and therefore the first case in the statement is a consequence of \eqref{minmonotono}.

Hence we now focus on the case in which~$W$ is differentiable in $\pm 1$.
For any~$\epsilon>0$, we consider the function 
\[v_\eps(x):= \begin{cases} -1 & {\mbox{ if }}x<0,\\ 
1-\eps  &{\mbox{ if }} 0<x<2r, \\ 1 &{\mbox{ if }}x>2r.\end{cases} \]
Then
\[\mathcal{E}(v_\eps)= \int_{-r}^r \frac{(2-\eps)^2}{2r^2}\, ds
+\int_{r}^{3r} \frac{\eps^2}{2r^2}\,ds
+\int_0^{2r} W(1-\eps)\,
ds= \frac{4}{r}+\eps\left(\frac{2\eps}{r}-\frac{4}{r}+2r\frac{W(1-\eps)}{\eps}\right).\]
Since $W$ is differentiable in $1$, recalling that $W(1)=0=W'(1)$ (thanks to~\eqref{DOUBLEW}), 
we get that \[\lim_{\eps\to 0}\left(\frac{2\eps}{r}-\frac{4}{r}+2r\frac{W(1-\eps)-W(1)}{\eps}\right)
= -\frac{4}{r}<0,\]
and therefore
there exists $\eps_0=\eps_0(r)$ such that for all $0<\eps<\eps_0$
we get that
\[ \mathcal{E}(v_\eps)=\frac{4}{r}+\eps\left(\frac{2\eps}{r}
-\frac{4}{r}+2r\frac{W(1-\eps)-W(1)}{\eps}\right)<\frac{4}{r}=\mathcal{E}(u_0).\]
This completes the proof of Proposition \ref{propminimi}.
\end{proof}

\section{Rigidity results for minimal heteroclinic connections,
and proof of Theorem~\ref{MONOTH}}\label{secmon}

We divide the proof of Theorem~\ref{MONOTH} in several steps.
{F}rom now on, we assume that~$u$ is as in the statement of
Theorem~\ref{MONOTH}.
\medskip 

\noindent{\bf Step 1: bounds on $u$, namely  $|u|\le1$.}

We fix~$\delta>0$
and we show that
\begin{equation}\label{8i819igfdh-293-1234985}
u\le 1+\delta\qquad{\mbox{(up to null measure sets).}}\end{equation}
To this end, we argue by contradiction and assume, for instance,
that the set~$\{u>1+\delta\}$ has positive measure.
Let~$v:=\min \{u,1+\delta\}$. By~\eqref{LIMITS}, we know that
there exist~$\alpha_0$, $\beta_0\in\R$ such
that~$u\le0$ in~$(-\infty,\alpha_0]$
and~$|u-1|\le\frac\delta2$ in~$[\beta_0,+\infty)$.
In particular, if~$\alpha:=\alpha_0-r$ and~$x\in(-\infty,\alpha+r)$, we have
that~$u(x)\le0$ and so~$u(x)=v(x)$. Also, if~$\beta:=\beta_0+r$
and~$x\in(\beta-r,+\infty)$, then~$u(x)\le 1+\frac\delta2$
and so~$u(x)=v(x)$. These considerations give that
\begin{equation}\label{tryeia96576y}
{\mbox{$u=v$
outside~$[\alpha+r,\beta-r]$}}\end{equation}
and so, by minimality,
\begin{equation}\label{MI0NU0NI02}
\begin{split}
0\,&\le
{\mathcal{E}}_{(\alpha,\beta)}(v)-{\mathcal{E}}_{(\alpha,\beta)}(u)
\\&=\frac1{2r^2}\int_\alpha^\beta \left(\osc{(x-r,x+r)} v\right)^2\,dx
-\frac1{2r^2}\int_\alpha^\beta \left(\osc{(x-r,x+r)} u\right)^2\,dx\\
&\qquad \qquad +\int_{(\alpha+r,\beta-r)\cap\{u>1+\delta\}} \Big( W(1+\delta)-W(u(x))\Big)\,dx.
\end{split}
\end{equation}
We also remark that, by \eqref{taglio},  for any~$x\in\R$,
\begin{equation}\label{DEC-o}
\osc{(x-r,x+r)} v\le\osc{(x-r,x+r)} u.
\end{equation}
Then, by~\eqref{MI0NU0NI02} and~\eqref{DEC-o},
$$ 0\le\int_{(\alpha+r,\beta-r)\cap\{u>1+\delta\}} \Big( W(1+\delta)-W(u(x))\Big)\,dx.$$
Recalling the assumptions on $W$ in~\eqref{DOUBLEW},
we conclude that~$(\alpha+r,\beta-r)\cap\{u>1+\delta\}$
must have zero Lebesgue measure.
Also, by~\eqref{tryeia96576y}, we have that~$u=v\le1-\delta$
outside~$[\alpha+r,\beta-r]$.
We thereby obtain that~$\{u>1+\delta\}$ has zero Lebesgue measure,
which proves~\eqref{8i819igfdh-293-1234985}.
Then, since~$\delta$ can be taken arbitrarily close to zero
in~\eqref{8i819igfdh-293-1234985}, we infer that~$u\le1$.

In a similar manner, one shows that~$u\ge-1$, and then
the claim follows, as desired.

\medskip

\noindent{\bf Step 2:  $u$ has finite global   energy.}

Namely, we show here that \begin{equation} \label{E BOUND}
 \mathcal{E}(u):=\frac1{2r^2}\int_\R \left(\osc{(x-r,x+r)} u\right)^2\,dx+\int_{\R} W(u(x))\,dx<+\infty.\end{equation}
For this, we fix~$R\geq 2(r+1)$ and let~$\xi_R\in C^\infty(\R,[0,1])$,
with~$\xi_R=1$ in~$[-R+1,R-1]$, $\xi_R=0$ in~$(-\infty, -R]\cup[R, +\infty)$ and~$|\xi_R'|\le 4$.
Let~$u_R:=\xi_R+(1-\xi_R) u$. Notice that~$u_R=u$ outside~$[-R,R]$,
and so the minimality of~$u$ gives that
$$ {\mathcal{E}}_{(-R-r,R+r)}(u)\le {\mathcal{E}}_{(-R-r,R+r)}(u_R),
$$
namely
\begin{equation}\label{bgby8}
\begin{split}&
\frac1{2r^2}\int_{-R-r}^{R+r} \left(\osc{(x-r,x+r)} u\right)^2\,dx
+\int_{-R-2r}^{R+r} W(u(x))\,dx \\
\leq\; &
\frac1{2r^2}\int_{-R-r}^{R+r} 
\left(\osc{(x-r,x+r)} u_R\right)^2 \,dx+\int_{-R-r}^{R+r}  W(u_R(x)) \,dx.
\end{split}\end{equation}
Now, if~$x\in ( -R+1+r,R-1-r)$, we have that~$(x-r,x+r)\subseteq( -R+1,R-1)$,
where~$\xi_R=1$ and so~$u_R=1$. Consequently,
\begin{equation}\label{MES}
{\mbox{for any $x\in ( -R+1+r,R-1-r)$, we have that }}\osc{(x-r,x+r)} u_R=0.
\end{equation}
On the other hand, by {\bf Step 1} and the definition of $u_R$, 
we have that~$|u_R|\leq 1$, and therefore
$$  \osc{(x-r,x+r)} u_R\le 2.$$
Combining this and~\eqref{MES}, we deduce that
\begin{equation}\label{bgby7}
\frac1{2r^2}\int_{-R-r}^{R+r} 
\left(\osc{(x-r,x+r)} u_R\right)^2\,dx\le
\frac1{2r^2}\int_{\{|x|\in(R-1-r,R+r)\}} 
\left(\osc{(x-r,x+r)} u_R\right)^2\,dx\le
\frac{4(1+2r)}{r^2}.
\end{equation}
Similarly, since~$u_R=1$ in~$( -R+1,R-1)$,
$$ \int_{-R-r}^{R+r} W(u_R(x))\,dx=
\int_{\{|x|\in(R-1,R+r)\}} W(u_R(x))\,dx\le 2(1+r)\|W\|_{L^\infty([-1,1])}.$$
Then, we plug this information and~\eqref{bgby7} into~\eqref{bgby8}
and we conclude that
\[ \frac1{2r^2}\int_{-R-r}^{R+r} \left(\osc{(x-r,x+r)} u\right)^2\,dx
+\int_{-R-r}^{R+r} W(u(x))\,dx\le\frac{4(1+2r)}{2r^2}+2(r+1)\|W\|_{L^\infty([-1,1])}.\]

By taking~$R$ as large as we wish, one deduces~\eqref{E BOUND},
as desired.

\medskip 

\noindent{\bf Step 3: monotonicity of $u$.} 

We suppose, by contradiction, that $u$ is not monotone, and so in
particular there exist $a$, $b\in\R$ Lebesgue points for~$u$ such that
\begin{equation}\label{aoboaobo}
a<b\qquad \mbox{ and }\qquad 
-1<B=\liminf_{x\to b}u(x)<\limsup_{x\to a}u(x)=A<1.\end{equation}
Our aim is to show that  quadruples $(a,b,A,B)$ which satisfy \eqref{aoboaobo} cannot exist, due to minimality of $u$. The argument used here is very similar to the one used in the proof of Lemma~\ref{competit}. 

First of all we claim the following: if quadruples  $(a,b,A,B)$ as in \eqref{aoboaobo} do exist, then 
\begin{equation}\label{claim1}  
\mbox{there exists at least one quadruple $(a,b,A,B)$ as in \eqref{aoboaobo} such that $B<0$.}\end{equation}
By contradiction, if it were not the case, we let $(a,b,A,B)$ any quadruple such that \eqref{aoboaobo} holds, with
\begin{equation}\label{abuno}
1>A>B\geq 0.\end{equation} 
In particular, since \eqref{claim1} does not hold,
necessarily
\begin{equation}\label{freytu54hbjfdb}
{\mbox{$u(x)\geq 0$ for almost every $x\in [a, +\infty)$.}}
\end{equation}
We also notice that, in light of~\eqref{abuno}, and recalling the assumptions in~\eqref{DOUBLEW},
we get that $W(t)>W(A)$ for every~$t\in [0,A)$. 
Now, we define the function 
\begin{equation}\label{tagliomax} v(x):=\begin{cases}  u(x) &{\mbox{ if }} x\leq a,\\ 
\max\{u(x), A\}& {\mbox{ if }}x>a. \end{cases} \end{equation}
Since~$A<1$ and~\eqref{LIMITS} holds true, we see that
there exists $c>a$ such that $v(x)=u(x)$ on $[c, +\infty)$.
Also, due to \eqref{taglio}, we get that 
$$\osc{(x-r,x+r)} v\leq \osc{(x-r,x+r)} u$$ for all $x$.
Moreover, thanks to~\eqref{freytu54hbjfdb}, we have that~$0\le u(x)\le v(x)=A<1$ for any~$x\in(a,c)$. Hence, by~\eqref{DOUBLEW},
we obtain that~$W(v(x))\leq W(u(x))$ for any~$x\in(a,c)$,
with strict inequality on the set
$$\{x\in (a,c) \;{\mbox{ s.t. }}\; u(x)<A\}$$
which has positive measure.

Collecting all these pieces of information, we conclude that
$$ {\mathcal{E}}_{(a-r,c+r)}(v)<{\mathcal{E}}_{(a-r,c+r)}(u),$$
and~$v$ is a competitor for~$u$ in~$(a-r,c+r)$, since~$v=u$
on~$(-\infty, a]\cup [c, +\infty)$.
This is in contradiction with the
minimality of $u$, and therefore~\eqref{claim1} is established. 
 
As a consequence, we fix now a quadruple $(a,b_0,A,B_0)$ as in \eqref{claim1}, with $-1<B_0<0$, and we define 
\begin{equation}\label{defa0} A_0:=\sup_{x\leq b_0}\limsup_{y\to x} u(y).\end{equation} 
Then~$A_0\geq A>B_0$.
By definition of $A_0$ there  exists a sequence~$
\eta_{j}\in(-\infty,b_0]$ with
\[
{\mbox{$u(\eta_{j})\to A_0$ as~$j\to+\infty$.}}\]
We observe that, since~$A_0>-1$ and~\eqref{LIMITS} holds true,
the sequence
$\eta_{j}$ is uniformly bounded, otherwise, passing to a subsequence, we would have $\eta_j\to -\infty$ and $u(\eta_j)\to -1\not =A_0$. 
So, passing to a subsequence, we define 
\begin{equation}\label{defa00}
a_0:=\lim_j \eta_j<b_0.\end{equation}

We claim that
\begin{equation}\label{claim2} A_0<1.\end{equation}
Not to interrupt this calculation, we postpone the proof of this claim, which is quite long, to {\bf Step 4}. 
 
We prove now that 
\begin{equation}\label{claim3} A_0+ B_0>0.\end{equation}
Assume on the contrary that $A_0+B_0\leq 0$.  If this is true, since $-1<B_0\leq 0$ and $B_0<A_0<1$, due to assumption \eqref{DOUBLEW}, we get that 
$W(t)>W(B_0)$ for all $t\in (B_0, A_0)$. We define the function 
\begin{equation}\label{tagliomin} v(x)=\begin{cases}  u(x) & {\mbox{ if }}
x\geq b_0,\\ 
\min\{u(x), B_0\}& {\mbox{ if }} x<b_0. \end{cases} \end{equation}
We observe that, since~$B_0>-1$ and~\eqref{LIMITS} holds true,
there exists $c_0<0$ such that~$v(x)=u(x)$ on~$(-\infty, c_0]$. 
Moreover, by the definition of $A_0$ in~\eqref{defa0}, we get that~$u(x)\leq A_0$ for almost every  $x\leq b_0$. Therefore,
as shown before, we get that $W(v(x))\leq W(u(x))$
for any~$x\in(c_0, b_0)$, with strict inequality on the set
$$\{x\in (c_0, b_0)\;{\mbox{ s.t. }}\; B_0< u(x)<A_0\}$$
which has positive measure. 
Therefore $v=u$ on $(-\infty, c_0]\cup [b_0, +\infty)$ and has strictly less potential energy in $(c_0-r,b_0+r)$.
These observations contradict the minimality of $u$, and so~\eqref{claim3} holds true.

We define now 
\begin{equation}\label{defb1} B_1:=\inf_{x\geq a_0}\liminf_{y\to x} u(y)\leq B_0.\end{equation} 
We show that 
\begin{equation}\label{claim5} B_1<B_0.\end{equation} Indeed, if this were not the case, then, in light of~\eqref{defb1}, we have that $B_1=B_0$ and $u(x)\geq B_0$ for almost every $x\geq a_0$. We note that
since $A_0+B_0>0$ by \eqref{claim3}, and~$B_0<0$, then $W(t)>W(A_0)$ for every~$t\in [B_0, A_0)$. 
We define the function $v$ as in \eqref{tagliomax} with $A_0$ in place of $A$ and $a_0$ in place of $a$. 
Again, since~$A_0<1$ by \eqref{claim2} and~\eqref{LIMITS} holds true,
we have that there exists $c_1>a_0$ such
that~$v(x)=u(x)$ on $[c_1, +\infty)$.
Moreover, due to \eqref{taglio}, we get that
$$\osc{(x-r,x+r)} v\leq \osc{(x-r,x+r)} u$$ for all $x$,
and, as shown before, we see that~$W(v(x))\leq W(u(x))$
with strict inequality on the set
$$\{x\in (a,c)\, {\mbox{ s.t. }}\, u(x)<A_0\}$$
which has positive measure.
Therefore $v=u$ on $(-\infty, a_0]\cup [c_1, +\infty)$ and has strictly less potential energy in $(a-r,c+r)$:
by the minimality of $u$, we find that necessarily $u=v$.
Therefore~\eqref{claim5} holds true. 

Now, by the definition of $B_1$ in~\eqref{defb1},
there exists a sequence~$
\zeta_{j}\in[a_0, +\infty)$ with
\[
{\mbox{$u(\zeta_{j})\to B_1$ as~$j\to+\infty$.}}\]
We observe that, due to the fact that \eqref{LIMITS} holds and $B_1<1$,  the sequence
{$\zeta_{j}$ is uniformly bounded, and  passing to a subsequence, we define 
\[b_1:=\lim_j \zeta_j>a_0.\]   
Following the same arguments as  for the proof of \eqref{claim2}, we can prove that 
\begin{equation}\label{claim4} B_1>-1.\end{equation}  
See {\bf Step 5} for a brief sketch of this. 
 
Next we observe that \begin{equation}\label{claim6} A_0+B_1<0.
\end{equation} 
Indeed, if this were not true, we can argue as in the 
the proof of claim \eqref{claim5}, define the function $v$ as in \eqref{tagliomax} with $A_0$ in place of $A$ and $a_0$ in place of $a$, and show that $u=v$ outside a compact interval and moreover that $v$ has strictly less energy of $u$, since $u(x)\geq B_1$ for almost every $x\geq a_0$, in contradiction with the minimality of~$u$. 

Then, we claim that 
\begin{equation}\label{claim7} b_1>b_0. 
\end{equation} 
Indeed, if this were not the case, then $a_0<b_1<b_0$,
and in particular $u(x)\leq A_0$ for every $x\leq b_1$, and~$A_0+B_1<0$, by \eqref{claim6}.  Hence, we can proceed as in 
the proof of claim \eqref{claim3}. That is, briefly,
we define the function $v$ as in \eqref{tagliomin} with $B_1$ in place of $B_0$ and $b_1$ in place of $b_0$, we show that $u=v$ outside a compact interval and finally we prove that $v$ has strictly less energy of $u$, since $u(x)\leq A_0$ for almost every $x\leq b_1$, in contradiction with the minimality of~$u$.
 
Now we define  
\begin{equation}\label{defa1} A_1:=\sup_{x\leq b_1}\limsup_{y\to x} u(y).\end{equation} 
Then $A_1\geq A_0>B_1$. Also, we observe that $A_1>A_0$, otherwise we could repeat exactly the same proof of claim \eqref{claim7} and obtain a contradiction to the minimality of $u$.
Moreover, we see that~$A_1<1$ by using the same argument as for \eqref{claim2}, see {\bf Step 4}. 

By definition of $A_1$ in~\eqref{defa1},
there  exists a sequence~$
\eta_{j}\in(-\infty,b_1]$ with
\[
{\mbox{$u(\eta_{j})\to A_1$ as~$j\to+\infty$.}}\]
Up to passing to a subsequence we define \[a_1:=\lim_j \eta_j.\] 
Since $A_1>A_0$ and $u(x)\leq A_0$ for almost every  $x\leq b_0$, necessarily $ a_0<b_0<a_1<b_1$.  Moreover $u(x)\geq B_1$ for almost every $x\geq a_1$ and 
$u(x)\leq A_1$ for almost every $x\leq b_1$. 

We observe that two possibilities may arise: either $A_1+B_1<0$ or $A_1+B_1\geq 0$. We will show that both of them are in contradiction with the minimality of $u$ and then this implies that 
quadruples as in \eqref{aoboaobo} cannot exist, and then finally that $u$ is monotone. 

If $A_1+B_1\geq 0$,  we argue as in the 
the proof of claim \eqref{claim5}, namely we define the function $v$ as in \eqref{tagliomax} with $A_1$ in place of $A$ and $a_1$ in place of $a$, and we show that $u=v$ outside a compact interval and moreover that $v$ has strictly less energy of $u$, since $u(x)\geq B_1$ for almost every $x\geq a_1$, in contradiction to the minimality of~$u$.

If instead~$A_1+B_1<0$, we can proceed as in 
the proof of claim \eqref{claim3}: we  define the function $v$ as in \eqref{tagliomin} with $B_1$ in place of $B_0$ and $b_1$ in place of $b_0$, we show that $u=v$ outside a compact interval and finally we prove that $v$ has strictly less energy of $u$, since $u(x)\leq A_1$ for almost every $x\leq b_1$, in contradiction to the minimality of~$u$.

These observations imply that~$u$ is monotone, and thus the claim 
in {\bf{Step 3}} is established.
\medskip 

\noindent{\bf Step 4: proof of claim \eqref{claim2}.}

We argue towards a contradiction, assuming that~$A_0=1$.
Hence, recalling~\eqref{defa0} and~\eqref{defa00}, we have that
\begin{equation}\label{986hggwp}
\limsup_{x\to a_0} u(x)=1.\end{equation}
Let ~$\mu\in(0,1)$, to be taken arbitrarily small
in the following. Then, by~\eqref{LIMITS},
we know that there exists~$\rho_\mu>a_0$ such that  
\begin{equation}\label{2.18BIS}
u(x)\geq 1-\mu \qquad {\mbox{ for every }}\,x\in[\rho_\mu, +\infty).\end{equation}
For~$\rho\ge\rho_\mu>a_0$,
we take~$\tau_{\rho}\in C^\infty \left(\R,[0,1]\right)$,
with~$\tau_{\rho}=1$ in~$(-\infty,\rho]$ and $\tau_{\rho}=0$ in~$[3r+\rho, +\infty)$.
Let also
\begin{equation}\label{urhoR} u_{\rho}(x):=\left\{\begin{matrix}
u(x) & {\mbox{ if }}x\le a_0,\\
\tau_{\rho}(x)+(1-\tau_{\rho}(x))\,u(x)&{\mbox{ if }}x>a_0.\end{matrix}
\right.\end{equation}
We notice that, since $\rho>a_0$,
we get that
\begin{equation}\label{jgnvcsaedwkergm:0}
u_{\rho}=1\quad  {\mbox{in }}\,(a_0, \rho].\end{equation} 
Furthermore, since $|u|\leq1 $ by {\bf Step 1},
and~$\tau_{\rho}\ge0$, we get that, if~$x>a_0$,
$$ u_\rho-u=\tau_{\rho}(1-u)\ge0,$$
and therefore
\begin{equation} \label{726359238882:1}
u\le u_{\rho}\leq 1. \end{equation} 
Moreover, for any~$x\in(a_0-r,a_0+r)$ we have that~$a_0\in(x-r,x+r)$,
and thus, in light of~\eqref{986hggwp},
\begin{equation} \label{726359238882:2}
\sup_{(x-r,x+r)} u=\sup_{(x-r,x+r)} u_{\rho}=1.\end{equation}
By~\eqref{726359238882:1} and~\eqref{726359238882:2} we obtain that,
for any~$x\in(a_0-r,a_0+r)$,
\begin{equation}\label{osc1} \osc{(x-r,x+r)} u_{\rho}\le \osc{(x-r,x+r)} u.\end{equation} 

Now, we observe that, by definition, $u_{\rho}=u$ outside~$[a_0,\rho+3r]$, so, by the minimality of~$u$, we get that
\begin{equation}\label{CRCR-110}
\begin{split}
0 \,&\le{\mathcal{E}}_{(a_0-r,\rho+4r)}(u_{\rho})-{\mathcal{E}}_{
(a_0-r,\rho+4r)}(u)\\
&=
\frac1{2r^2}\int_{a_0-r}^{\rho+4r} \left[\left(\osc{(x-r,x+r)} u_{\rho}\right)^2-
\left(\osc{(x-r,x+r)} u\right)^2\right]\,dx
+\int_{a_0-r}^{\rho+4r} \Big( W(u_{\rho}(x))-W(u(x))\Big)\,dx.
\end{split}\end{equation}
Hence, recalling \eqref{osc1} and the definition of $u_{\rho}$, we obtain 
\begin{equation}\label{CR-110}
0\le
\frac1{2r^2}\int_{a_0+r}^{\rho+4r} \left[\left(\osc{(x-r,x+r)} u_{\rho}\right)^2-
\left(\osc{(x-r,x+r)} u\right)^2\right]\,dx
+\int_{a_0}^{\rho+3r} \Big( W(u_{\rho}(x))-W(u(x))\Big)\,dx.
\end{equation}

Now, we claim that
\begin{equation}\label{j865ynfdldsa}
\frac1{2r^2}\int_{a_0+r}^{+\infty}
\left(\osc{(x-r,x+r)} u \right)^2\, dx+\int_{a_0}^{+\infty}W(u(x))dx>0.
\end{equation}
Indeed, if it were not the case, we would have that~$u(x)=1$ for almost every~$x\geq a_0$. But this would be in contradiction with the fact that $a_0<b_0$ and $\liminf_{x\to b_0} u(x)=B_0<0$. Hence,
\eqref{j865ynfdldsa} is established.

As a consequence of~\eqref{j865ynfdldsa}, for large~$\rho$, we can write
\begin{equation}\label{8u01939jxJAJ-xU}
\frac1{2r^2}\int_{a_0+r}^{\rho-r}
\left(\osc{(x-r,x+r)} u\right)^2\,dx
+\int_{a_0}^{\rho} W(u(x))\,dx\ge \hat c,\end{equation}
for some~$\hat c>0$, independent of~$\mu$ and~$\rho$.

Also, from~\eqref{jgnvcsaedwkergm:0} we deduce that
\begin{equation*}
\frac1{2r^2}\int_{a_0+r}^{\rho-r}
\left(\osc{(x-r,x+r)} u_{\rho}\right)^2dx=0
\qquad{\mbox{and}}\qquad
\int_{a_0}^{\rho}W(u_{\rho}(x))\,dx
=0.\end{equation*}
This and~\eqref{8u01939jxJAJ-xU} imply that
$$ 
\frac1{2r^2}\int_{a_0+r}^{\rho-r} \left[\left(\osc{(x-r,x+r)} u_{\rho}\right)^2-
\left(\osc{(x-r,x+r)} u\right)^2\right]\,dx
+\int_{a_0}^{\rho} \Big( W(u_{\rho}(x))-W(u(x))\Big)\,dx\le-\hat{c}.
$$
Then, we insert this information into~\eqref{CR-110} and we find that
\begin{equation}\label{6567826799728827171832}
\hat{c}\le
\frac1{2r^2}\int_{\rho-r}^{\rho+4r} \left[\left(\osc{(x-r,x+r)} u_{\rho,R}\right)^2-
\left(\osc{(x-r,x+r)} u\right)^2\right]\,dx
+\int_\rho^{\rho+3r} \Big( W(u_{\rho}(x))-W(u(x))\Big)\,dx.
\end{equation}

Now we observe that, thanks to~\eqref{2.18BIS}, for any~$x\in[\rho,\rho+3r]$,
$$ u_{\rho}(x)-u(x)=
\tau_{\rho}(x)\,(1-u(x))\leq\mu,$$ and thus
\begin{equation*}
\left|\int_\rho^{\rho+3r} \Big( W(u_{\rho}(x))-W(u(x))\Big)\,dx\right|\le
3 r \max_{{t,s\in [-1,1]}\atop{|t-s|\leq \mu}} |W(t)-W(s)|\to 0 \qquad \mbox{as $\mu\to 0$}.
\end{equation*}
Using this,
as long as~$\mu>0$ is sufficiently small (possibly in dependence on~$r$), we get
from \eqref{6567826799728827171832} that
\begin{equation}\label{878776766192777j1}
\frac{\hat{c}}{2}\le
\frac1{2r^2}\int_{\rho-r}^{\rho+4r} \left[\left(\osc{(x-r,x+r)} u_{\rho}\right)^2-
\left(\osc{(x-r,x+r)} u\right)^2\right]\,dx.
\end{equation}
 
We also observe that if~$\rho\geq \rho_\mu+2r$
and~$x\ge\rho-r$, then~$x-r\geq \rho_\mu>a_0$, and 
therefore, by the definition of~$u_\rho$,
and recalling~\eqref{2.18BIS}, we get 
\[\sup_{(x-r,x+r)} u_{\rho}\leq \sup_{(x-r,x+r)} u+\sup_{(x-r,x+r)} \tau_{\rho}(1-u)\leq \sup_{(x-r,x+r)} u+\mu.\]
Then, using this observation and recalling \eqref{726359238882:1},
we conclude that 
\begin{equation}\label{87-os01}
\osc{(x-r,x+r)} u_{\rho}\le \osc{(x-r,x+r)} u+\mu,
\end{equation} 
for any~$x\ge\rho-r$ with~$\rho\geq \rho_\mu+2r$.

{F}rom~\eqref{87-os01}, for any~$x\ge\rho-r$
and~$\rho\geq \rho_\mu+2r$, we have that
\begin{eqnarray*}&&
\left(\osc{(x-r,x+r)} u_{\rho}\right)^2-
\left(\osc{(x-r,x+r)} u\right)^2
\le\left(\osc{(x-r,x+r)} u+\mu\right)^2-
\left(\osc{(x-r,x+r)} u\right)^2\\
&&\qquad\quad= \mu^2+2\mu\,\osc{(x-r,x+r)} u
\leq \mu^2+\mu+\mu\left(\osc{(x-r,x+r)} u\right)^2.
\end{eqnarray*}
Therefore, we conclude that, if~$\rho\geq \rho_\mu+2r$, 
\[
\int_{\rho-r}^{\rho+4r} \left[\left(\osc{(x-r,x+r)} u_{\rho}\right)^2-
\left(\osc{(x-r,x+r)} u\right)^2\right]\,dx\le 5\mu^2r+5\mu r+ \mu  \mathcal{E}(u)
\] where $ \mathcal{E}(u)$ is the energy defined in~\eqref{E BOUND}.
Plugging this information into~\eqref{878776766192777j1},
and recalling the claim~\eqref{E BOUND} in {\bf Step 2}, we conclude that
\[ \frac{\hat{c}}{2} \le\frac{ 5\mu^2r+5\mu r+ \mu  \mathcal{E}(u)}{2r^2},\]
which leads to a contradiction by sending~$\mu\searrow0$, and this concludes the proof of \eqref{claim2}.  
 
 \smallskip 

\noindent{\bf Step 5: proof of claim \eqref{claim4}.}

For the proof of \eqref{claim4}, the argument is the same as for the proof of  \eqref{claim2} in {\bf Step 4}, with obvious modifications. We sketch it briefly for the reader's convenience.

We fix $\mu>0$ and $ \lambda_\mu<b_1$ such that $u(x)<-1+\mu$
for every $x\leq \lambda_\mu$.  For any~$\lambda<\lambda_\mu$, 
we take~$\tau_{\lambda}\in C^\infty \left(\R,[-1,0]\right)$,
with~$\tau_{\lambda}=-1$ in~$[\lambda, +\infty)$ and $\tau_{\lambda}=0$ in~$(-\infty, \lambda-3r,]$, and we define 
\begin{equation}\label{ul} u_{\lambda}(x):=\left\{\begin{matrix}
u(x) & {\mbox{ if }}x\ge b_1,\\
\tau_{\lambda}(x)+(1+\tau_{\lambda}(x))\,u(x)&{\mbox{ if }}x<b_1.\end{matrix}
\right.\end{equation} 
As done in {\bf Step 4}, it is easy to check that
for any~$x\in(b_1-r,b_1+r)$
\begin{equation}\label{osc11}
\osc{(x-r,x+r)} u_{\lambda}\le \osc{(x-r,x+r)} u\end{equation} 
and $u=u_\lambda$ outside $[\lambda-3r, b_1]$. As a consequence
of these observations and of the minimality of $u$, 
\[
0\le
\frac1{2r^2}\int_{\lambda-4r}^{b_1-r} \left[\left(\osc{(x-r,x+r)} u_{\lambda}\right)^2-
\left(\osc{(x-r,x+r)} u\right)^2\right]\,dx
+\int_{\lambda-3r}^{b_1} \Big( W(u_{\lambda}(x))-W(u(x))\Big)\,dx.
\]
As in {\bf Step 4}, we see that, for $\lambda<<-1$,
\[\frac1{2r^2}\int_{\lambda+r}^{b_1-r}
\left(\osc{(x-r,x+r)} u\right)^2\,dx
+\int_{\lambda}^{b_1} W(u(x))\,dx\ge \hat c,\]
for some~$\hat c>0$, independent of~$\mu$, $\lambda$, otherwise we would get $u(x)=-1$ for almost every $x\leq b_1$ in contradiction with the definition of $A_0$. 

Thus, using the fact that ~$u_{\lambda}=-1$ in~$[\lambda, b_1)$, we conclude that
\[
\hat{c}\le
\frac1{2r^2}\int_{\lambda-4r}^{\lambda+r} \left[\left(\osc{(x-r,x+r)} u_{\lambda}\right)^2-
\left(\osc{(x-r,x+r)} u\right)^2\right]\,dx
+\int_{\lambda-3r}^{\lambda} \Big( W(u_{\lambda}(x))-W(u(x))\Big)\,dx.
\]
Now we observe that, for any~$x\in[\lambda-3r, \lambda]$,
\begin{equation*}\left|
\int_{\lambda-3r}^\lambda \Big( W(u_{\lambda}(x))-W(u(x))\Big)\,dx\right|\le
3 r \max_{{t,s\in [-1,1]}\atop{|t-s|\leq \mu}} |W(t)-W(s)|\to 0 \qquad \mbox{as $\mu\to 0$}.
\end{equation*}
Therefore, for $\lambda<<-1$, we get that
\begin{equation}\label{pro}
\frac{\hat{c}}{2}\le
\frac1{2r^2}\int_{\lambda-4r}^{\lambda+r} \left[\left(\osc{(x-r,x+r)} u_{\lambda}\right)^2-
\left(\osc{(x-r,x+r)} u\right)^2\right]\,dx.
\end{equation}
Moreover, recalling the definition of~$u_\lambda$,
it is easy to check that,
for any $x\leq \lambda+r$ and~$\lambda<\lambda_\mu-2r$, there holds 
\[
\osc{(x-r,x+r)} u_{\lambda}\le \osc{(x-r,x+r)} u+\mu.
\]
Now
the conclusion follows plugging this information in \eqref{pro} and
sending $\mu\to 0$, obtaining a contradiction as in {\bf Step 4}.

\section{The difference equation satisfied by minimal heteroclinic connections,
and proof of Theorem~\ref{euler}} \label{seceuler} 

In this section, we provide a proof of Theorem \ref{euler}. In order to get the result, we will need to introduce an auxiliary functional. 
We observe that, for monotone functions, the oscillation defined in~\eqref{1.1bis}
reads as
\[\osc{(x-r,x+r)} u=|u(x+r)-u(x-r)|.\] 
Moreover, it is easy to check that for any~$v\in L^{\infty}_{{\rm{loc}}}(\R)$, there holds 
$$ \left|v(x+r)-v(x-r)\right|\le \osc{(x-r,x+r)} v.$$

We introduce the following 
auxiliary functional, defined, for any~$r>0$, $a<b$, and~$W$ as in~\eqref{DOUBLEW},
as
\begin{equation}\label{CALF}
{\mathcal{F}}_{(a,b)}(u):=\frac1{2r^2}\int_a^b \left(u(x+r)-u(x-r)\right)^2\,dx
+\int_{a}^{b} W(u(x))\,dx.\end{equation}
In this setting, we say that~$u\in L^{\infty}_{\rm{loc}}(\R)$ is a local minimizer
of~${\mathcal{F}}$ if, for any~$a<b$ and any~$v\in L^{\infty}_{\rm{loc}}(\R)$
such that~$u=v$ outside~$[a+r,b-r]$, we have that
$$ {\mathcal{F}}_{(a,b)}(u)\le {\mathcal{F}}_{(a,b)}(v).$$
It is easy to check that {if~$u$ is a critical point
for the operator ${\mathcal{F}}_{(a,b)}$ in~\eqref{CALF}, then
\begin{equation}\label{EL}
 \int_\R \frac{(u(x+2r)+u(x-2r)-2 u(x))}{r^2} \,\varphi(x)\,dx =\,\int_\R W'(u(x))\,\varphi(x)
 \,dx,
\end{equation}
for all~$\varphi\in C^\infty(\R)$,
such that~$\varphi\equiv0\ \mbox{ in } \R\setminus[a+r,b-r]$.

Comparing~\eqref{CALE} and~\eqref{CALF}, one notices that
 for any~$v\in L^{\infty}_{{\rm{loc}}}(\R)$,
\begin{equation}\label{ine}
{\mathcal{F}}_{(a,b)}(v)\le{\mathcal{E}}_{(a,b)}(v).
\end{equation} and moreover 
\begin{equation}\label{FCONE}
{\mbox{if $u$ is monotone in $(a-r,b+r)$, then ${\mathcal{F}}_{(a,b)}(u)=
{\mathcal{E}}_{(a,b)}(u)$.}}
\end{equation}

Combining \eqref{ine} with~\eqref{FCONE} we obtain that
\begin{equation}\label{min1}
\begin{split}&\mbox{
if $u$ is monotone and it is a local minimizer for the functional
in~\eqref{CALF},} \\ &\mbox{
then it is also a local minimizer for the functional in~\eqref{CALE}.}\end{split}\end{equation} 

In order to prove Theorem \ref{euler}, in light of~\eqref{EL}, it is sufficient to show that 
the reverse statement of~\eqref{min1} holds true. 
This will be accomplished in Proposition~\ref{teoincr} below.

To this aim, we need the following result,
which is the analogous of Proposition~\ref{exdir}
for the functional~\eqref{CALF}.
More precisely:

\begin{lemma}\label{competit2} 
Assume that~\eqref{DOUBLEW} holds true and consider  $R>2r$. 

 Then there exists a solution $v_R$ to the Dirichlet problem \begin{equation}\label{dirpb2}
\inf \Big\{ \mathcal{F}_{(-R,R)} (v) {\mbox{ s.t. }} v\in L^{\infty}(\R),
\, v(x)=1 \,\mbox{ for a.e. $x\geq R-r$ and }\,v(x)=-1\,\mbox{ for a.e. $x\leq -R+r$}\Big\}. 
\end{equation} 
Moreover $v_R$ is monotone nondecreasing and piecewise constant on intervals of length at least $2r$.

Finally $v_R$ is also a solution to the Dirichlet problem \eqref{dirpb}.
\end{lemma} 

\begin{proof}
 
 We notice that for any~$c\in \R$ and
for any~$v\in L^{\infty}_{{\rm{loc}}}(\R)$, we get that 
\begin{equation}\label{incrementi1} 
|v(x+r)-v(x-r)|=|\min\{v(x+r),c\}-\min\{v(x-r),c\}|+
|\max\{v(x+r),c\}-\max\{v(x-r),c\}|,
\end{equation}
at almost every~$x$.
Indeed, if~$v(x-r)\leq c$ and~$v(x+r)\leq c$, then
\begin{eqnarray*}
&& |\min\{v(x+r),c\}-\min\{v(x-r),c\}|=|v(x+r)-v(x-r)|,  \\
{\mbox{and }}&&
|\max\{v(x+r),c\}-\max\{v(x-r),c\}|=0,
\end{eqnarray*}
which implies~\eqref{incrementi1} in this case.
The case in which~$v(x-r)\geq c$ and~$v(x+r)\geq c$ can be treated
similarly.

If instead~$v(x-r)\leq c\leq v(x+r)$ (or similarly~$v(x+r)\leq c
\leq v(x-r)$), then\
\begin{eqnarray*} && |\min\{v(x+r),c\}-\min\{v(x-r),c\}|
+ |\max\{v(x+r),c\}-\max\{v(x-r),c\}|\\ &&\qquad
= c-v(x-r)+v(x+r)-c=|v(x+r)-v(x-r)|,\end{eqnarray*}
which gives~\eqref{incrementi1} in this case as well.

Therefore, thanks to \eqref{incrementi1} and \eqref{DOUBLEW}, we get that 
\[\mathcal{F}_{(-R,R)} \big(\max\{-1,\min\{1,v\}\}\big)
\leq \mathcal{F}_{(-R,R)} (v).\] 
This implies that we can reduce to consider the case in
which~$-1\leq v(x)\leq 1$ for a.e. $x$.  
So, we fix  $v\in L^\infty(\R)$ with values in $[-1,1]$  such that 
\begin{equation}\label{BOrs}
{\mbox{$v(x)=1$ for a.e. $x\geq R-r$ and $v(x)=-1$ for a.e. $x\leq -R+r$. 
}}\end{equation}
Let $N:=\left\lceil \frac{2R}{r}\right\rceil . $ 
For any fixed~$x\in [-R-r, -R)$,
we define the sequence
\begin{equation}\label{seq} v_n^x:= v(x+nr), \qquad\text{ for }n\in\Z.
\end{equation}
Note that 
\begin{equation}\label{fkerohin}
{\mbox{$v_n^x=-1$ if $n\le1$ and $v_n^x=1$ if $n\ge N$,
for all $x\in [-R-r, -R)$,}}\end{equation} thanks to~\eqref{BOrs}.

{F}rom~\eqref{BOrs}, we also see that
\begin{equation}\label{AUIJudjhascnikx}
\begin{split}
{\mathcal{F}}(v)\,&=
\int_{\R}\left(  \frac{1}{2r^2} (v(y+r)-v(y-r))^2+ W(v(y))\right)\,dy
\\&=\sum_{n\in\Z}
\int_{-R+(n-1)r}^{-R+nr}\left(  \frac{1}{2r^2} (v(y+r)-v(y-r))^2+ W(v(y))\right)\  dy.\end{split}
\end{equation}
Furthermore, by the change of variable $y=x+nr$, 
\[\int_{-R-r}^{-R}\left(  \frac{1}{2r^2} (v_{n+1}^x-v_{n-1}^x)^2+ W(v_n^x)\right)\  dx=
\int_{-R+(n-1)r}^{-R+nr}\left(  \frac{1}{2r^2} (v(y+r)-v(y-r))^2+ W(v(y))\right)\  dy. \]
{F}rom this and~\eqref{AUIJudjhascnikx}, dividing odd and even indexes,
we have that
\begin{eqnarray*}
\mathcal{F}_{(-R,R)} (v)&=&
\int_{-R-r}^{-R} \sum_{n\in\Z }\left(  \frac{1}{2r^2} (v_{n+1}^x-v_{n-1}^x)^2+ W(v_n^x)
\right)\  dx
\\&=&
\int_{-R-r}^{-R} \sum_{k\in\Z }\left(  \frac{1}{2r^2} (v_{2k+2}^x-v_{2k}^x)^2+ W(v_{2k+1}^x)
\right)\  dx\\&&\qquad+
\int_{-R-r}^{-R} \sum_{k\in\Z }\left(  \frac{1}{2r^2} (v_{2k+1}^x-v_{2k-1}^x)^2+ W(v_{2k}^x)
\right)\  dx.
\end{eqnarray*}
Hence, in light of~\eqref{fkerohin}, if we set $K:=\left\lfloor\frac{N}{2}\right\rfloor$,
we can write
\begin{equation}\begin{split}\label{discreteen1}
\mathcal{F}_{(-R,R)} (v)\,&=
\int_{-R-r}^{-R} \sum_{k=0 }^K\left(  \frac{1}{2r^2} (v_{2k+2}^x-v_{2k}^x)^2+ W(v_{2k+1}^x)
\right)\  dx\\&\qquad+
\int_{-R-r}^{-R} \sum_{k=0}^K\left(  \frac{1}{2r^2} (v_{2k+1}^x-v_{2k-1}^x)^2+ W(v_{2k}^x)
\right)\  dx.
\end{split}
\end{equation}
Then, we consider the finite minimization problem:
\begin{equation}\label{dirdiscreto}
m_R:=\min \left\{ \sum_{j=0}^K \left(\frac{1}{2r^2} (w_{j+1}-w_{j})^2+ W(w_j)\right)
\right\},
\end{equation}
where the class of competitors is such that~$w_j=-1$ for all~$j\le0$,
$w_{j}=1$ for all~$j\ge K+1$, and~$w_j\in[-1,1]$ for all~$j\in\Z$. 

Note that, by \eqref{fkerohin} and \eqref{discreteen1}, we get that 
 \begin{equation}\label{energiadisstima} 
 \mathcal{F}_{(-R,R)} (v)\geq 2\, r\,m_R\,.
\end{equation} 
%where $m_R$ is defined in \eqref{dirdiscreto}. 

With analogous arguments as in the proof of Lemma  \ref{competit}, one can see
that monotone sequences make the energy functional lower,
and therefore, as in Proposition~\ref{exdir}, one finds
that
there exists a monotone nondecreasing  solution to \eqref{dirdiscreto},
that is a solution with $w_j\leq w_{j+1}$.  Let us denote with $(w^R_j)$ this solution. 

We define now a function $v_R:[-R-r, R+r]\to [-1,1]$ as follows: \begin{equation}\label{PALCONS}
v_R(x):= w_j^R, \qquad \forall x\in [-R-r+2jr, -R-r+2(j+1)r).\end{equation}
We have that $v_R(x)=-1$ for $x\leq -R+r$, $v_R(x)=1$ for $x\geq R-r$, $v_R$ is monotone nondecreasing and 
 $  \mathcal{F}_{(-R,R)} (v_R)=2\; r\;m_R$. Recalling \eqref{energiadisstima},
 we see that~$v_R$ attains the minimal possible value, and this concludes the proof of existence of a monotone nondecreasing, piecewise constant solution to 
 \eqref{dirpb2}. 
 
Finally, $v_R$ is also a solution to the Dirichlet problem \eqref{dirpb}, due to \eqref{min1}. 
\end{proof}

As a consequence of Lemma \ref{competit2}, we obtain the counterpart of \eqref{min1}.

\begin{proposition}\label{teoincr}
Let $u$ be a local minimizer for the functional
in~\eqref{CALE} which satisfies \eqref{LIMITS}. Then, $u$
is also a local minimizer for the functional in~\eqref{CALF}.
\end{proposition}

\begin{proof} 
By Theorem \ref{MONOTH}, we know that $u$ is monotone.
Now, we argue by contradiction, assuming that there exist~$M_0>0$,
a function~$v\in L^{\infty}(\R)$ such that~$v=u$
outside~$[-M_0+r,M_0-r]$, and~$\eps>0$, such that
\begin{equation}\label{contraincr1} \mathcal{F}_{(-M_0, M_0)}(v)\leq \mathcal{F}_{(-M_0, M_0)}(u)-
2\eps= \mathcal{E}_{(-M_0, M_0)}(u)-
2\eps,\end{equation} 
where~\eqref{FCONE} is used in the last equality.

Since $u$ is a local minimizer for the functional \eqref{CALE}, and the two functionals \eqref{CALE} and \eqref{CALF} coincide on monotone functions, we get that $v$ is not monotone in $(M_0+r,M_0-r)$.  

Now we take~$M>M_0+r $ and we consider~$v^{M}$ and $u^M$ as given in \eqref{tagliov}.
Then, we get from~\eqref{contraincr1}   that
\[\mathcal{F}_{(-M, M)}(v^{M})\leq \mathcal{F}_{(-M, M)}(u^{M})-2\eps.\] 
Consequently, recalling the notation
of~$e_M$ and~$e$, as given in~\eqref{mongiu} and~\eqref{limer}, and exploiting~\eqref{cen} and~\eqref{stimavM},
we see that
\begin{equation}\label{c111}\begin{split}&
\mathcal{F}_{(-M, M)}(v^{M})\leq \mathcal{F}_{(-M, M)}(u^{M})-
2\eps= \mathcal{E}_{(-M, M)}(u^{M})-
2\eps\\ &\qquad\leq\mathcal{E}(u^M)-2\eps \leq  \mathcal{E}(u)-\eps=e-\eps\leq e_M-\eps
,\end{split}\end{equation}
where~$u_M$ is constructed in Proposition~\ref{exdir}.

Now, by Lemma \ref{competit2}, we get that there
exists $v_M$ which is monotone  and
such that 
\begin{equation}\label{c112X}e_M
=\mathcal{E}_{(-M,M)} (v_M)=\mathcal{F}_{(-M,M)} (v_M) 
.\end{equation}
On the other hand, since
$v_M=v^M=-1$ in $(-\infty, -M+r)$
and~$v_M=v^M=1$ in $(M-r, +\infty)$, we have that
$$ \mathcal{F}_{(-M,M)} (v_M)\leq \mathcal{F}_{(-M,M)} (v^M) $$
This and~\eqref{c112X} lead to
\begin{equation}\label{c112}e_M\le \mathcal{F}_{(-M,M)} (v^M) .
\end{equation}
{F}rom \eqref{c111} and \eqref{c112}  we get that
\[ e_M\leq \mathcal{E}_{(-M,M)}(u_M)-\eps=e_M-\eps\] which gives a contradiction. 
\end{proof}

{F}rom Proposition~\ref{teoincr}
we obtain Theorem \ref{euler} by arguing as follows:

\begin{proof}[Proof of Theorem \ref{euler}] By Proposition~\ref{teoincr}, we get that $u$ is a local minimizer of~\eqref{CALF}. Therefore, in view of~\eqref{EL}, 
it is a solution to \eqref{el1}. 
\end{proof} 

\section{Asymptotics as $r\searrow0$, and proof of Proposition~\ref{ASY}}\label{ASYS}

In this section, we show that the heteroclinic
connections constructed in this paper approach, for small~$r$,
the classical heteroclinics arising in ordinary differential
equations, as stated in Proposition~\ref{ASY}.

\begin{proof}[Proof of Proposition~\ref{ASY}]
Let $u_r$ be as in Theorem~\ref{exthm}, for a given~$r>0$.
Up to a translation, we will suppose that
\begin{equation}\label{5.1}
{\mbox{$u_r(x)<0$ for all~$x<0$ and $u_r(x)\ge0$ for all~$x>0$.}}
\end{equation}
Since~$u_r$ is monotone and bounded uniformly in~$r$,
and therefore bounded in BV uniformly in~$r$,
up to a subsequence we may suppose that~$u_r$ converges
to some~$u$ a.e. in~$\R$.

Consequently, from~\eqref{el1}, for every~$\phi\in C^\infty_0(\R)$,
\begin{eqnarray*}
0&=&\lim_{r\searrow0} \int_\R \left(
\frac{u_r(x+2r)+u_r(x-2r)-2u_r(x)}{r^2}-W'(u_r(x))
\right)\,\phi(x)\,dx\\
&=&\lim_{r\searrow0} 
\frac{1}{r^2}\left(
\int_\R u_r(x)\,\phi(x-2r)\,dx+
\int_\R u_r(x)\,\phi(x+2r)\,dx-2
\int_\R u_r(x)\,\phi(x)\,dx\right)
-\int_\R W'(u(x))\,\phi(x)\,dx\\
&=&
\lim_{r\searrow0} \int_\R u_r(x)\,\left(
\frac{\phi(x+2r)+\phi(x-2r)-2\phi(x)}{r^2}
\right)\,dx
-\int_\R W'(u(x))\,\phi(x)\,dx\\&=&
4\int_\R u(x)\,\phi''(x)\,dx
-\int_\R W'(u(x))\,\phi(x)\,dx.
\end{eqnarray*}
This gives that
\begin{equation*}
4u''=W'(u),\end{equation*}
in the distributional sense, and hence in the smooth sense as well,
and this proves~\eqref{EQ:44}.

Also, passing to the limit in~\eqref{5.1},
\begin{equation*}
{\mbox{$u(x)\le0$ for all~$x<0$ and $u(x)\ge0$ for all~$x>0$,}}
\end{equation*}
which gives that
$$ 
0\ge\lim_{x\nearrow0} u(x)=
u(0)=\lim_{x\searrow0} u(x)\ge0,$$
and hence~\eqref{EQ:45} is established.

Now, in light of~\eqref{minmonotono}, we can write that
\begin{eqnarray*}
&& 
\int_\R \left[ \frac1{2r^2}\Big(u(x+r)-u(x-r)\Big)^2+
W(u_r(x))\right]\,dx
=
\int_\R \left[ \frac1{2r^2}\left(\osc{(x-r,x+r)} u_r\right)^2+
W(u_r(x))\right]\,dx\\&&\qquad
=\mathcal{E}(u_r)\leq  \min\left\{\frac{4}{r}, 4+c_W\right\}
\le 4+c_W.\end{eqnarray*}
Accordingly, using Fatou's Lemma,
\begin{equation}\label{IKSBbdascjidwe234UJBS}
4+c_W \ge 
\int_\R \liminf_{r\searrow0}\left[ \frac{\big(u(x+r)-u(x-r)\big)^2}{2r^2}+
W(u_r(x))\right]\,dx=
\int_\R \left[ 2|u'(x)|^2+
W(u(x))\right]\,dx.
\end{equation}
Furthermore, we know that~$u$ is monotone nondecreasing
and with values in~$[-1,1]$, since so is~$u_r$, hence we can define
\begin{equation}\label{LIM:PA} \ell_\pm :=\lim_{x\to\pm\infty} u(x).\end{equation}
We claim that
\begin{equation}\label{AJM9iaksx}
\ell_+=1.
\end{equation}
Indeed, suppose not, say~$\ell_+<\ell$, for some~$\ell<1$. {F}rom~\eqref{EQ:45} and the monotonicity
of~$u$, we know that~$\ell_+\ge u(0)=0$. Therefore, there exists~$X>0$
such that for all~$x\ge X$ we have that~$u(x)\in \left[-\frac12,\ell\right]$.
This and~\eqref{DOUBLEW} give that, for all~$x\ge X$,
$$ W(u(x))\ge \inf_{\tau\in [-1/2,\ell]} W(\tau)>0,$$
and accordingly
$$ \int_X^{+\infty} W(u(x))\,dx=+\infty.$$
This is in contradiction with~\eqref{IKSBbdascjidwe234UJBS}, and hence it
proves~\eqref{AJM9iaksx}.

Similarly, one can show that~$\ell_-=-1$. This and~\eqref{AJM9iaksx}
lead to~\eqref{EQ:46}, as desired. Finally, since the heteroclinic satisfying \eqref{EQ:44}, \eqref{EQ:45}, \eqref{EQ:46} is unique, we have full convergence of $u_r$ to $u$.
\end{proof}

\section{Piecewise constant  heteroclinic connections,
and proof of Theorem~\ref{costante}}\label{secdiscrete} 

In this section, we prove Theorem~\ref{costante}.
The proof is based on the existence of  piecewise constant solutions to the Dirichlet problem \eqref{dirpb} given  in Lemma \ref{competit2}. 

\begin{proof}[Proof of Theorem \ref{costante}] 
By Lemma \ref{competit2}, we have that for all $R>2r$ there exists a monotone solution  $v_R$ to the Dirichlet problem \eqref{dirpb}, which is piecewise constant. 

%%% Arguing as in the proof of Theorem \ref{exthm}, we get that, u
Thus, up to  a translation, we can assume that 
$v_R(x)<0$ for any~$x<0$ and $v_R(x)\geq 0$ for any~$x>0$. 

Recalling  the construction in~\eqref{PALCONS}, we get that there exists $n(R)\to +\infty$ as $R\to +\infty$ such that
\begin{equation}\label{FVA:1}
{\mbox{$v_R$ is constant in intervals of the form $[2nr, 2(n+1)r)$ for all $-n(R)\leq n\leq n(R)$. }}\end{equation}
%%% By the same arguments as in the proof of Theorem \ref{exthm}, 
Since~$v_R$ is equibounded
and has equibounded total variation, by compactness
(see \cite[Corollary 3.49]{MR1857292}) 
up to extracting a subsequence,
we can define \begin{equation}\label{FVA:2}v(x):=\lim_{R\to+\infty} v_R(x),\end{equation}
where the limit holds
pointwise and locally in $L^p$ for every $p\geq 1$.

{F}rom this (see \eqref{LOCAM}), we also
conclude that $v$ is a local minimizer 
of \eqref{CALE}, which satisfies \eqref{LIMITS}. Finally, by~\eqref{FVA:1}
and~\eqref{FVA:2}, we obtain that~$v$ is constant on every interval $[2nr, 2(n+1)r)$, for $n\in \Z$. 
\end{proof} 
  
\section{Nonuniqueness issues, and proofs of Proposition \ref{corodis}
and Corollary \ref{coro1}}\label{NOAM:S}

Now, we provide the proof of Proposition \ref{corodis},  based on a   analogous argument as in the proof of Lemma~\ref{competit2}, Lemma~\ref{competit} and Theorem~\ref{exthm}. 
 \begin{proof}[Proof of Proposition \ref{corodis}] 
Fixed~$K\in\N$, we consider two finite minimization problems. The first one is the following 
\begin{equation}\label{dirdiscreto1}
m^{(1)}_K:=\min \left\{ \sum_{j=-K}^K \left(\frac{1}{2r^2} (w_{j+1}-w_{j})^2+ W(w_j)\right)
\right\},
\end{equation}
where the class of competitors $(w_j)$ is such that   $w_0=0$,  $w_j=-w_{-j}$ and $w_{j}=1$ for all~$j\ge K$. 

We note  that, for all $K>1$,
\begin{equation}\label{NANAAA}
m^{(1)}_K\leq\frac{1}{r^2}+W(0).\end{equation}
Indeed we choose the competitor $w_0:=0$ and $w_j:=1$ for all $j>0$,
and then~\eqref{NANAAA} plainly follows from~\eqref{dirdiscreto1}.

Moreover, by \eqref{dirdiscreto1}, we see that
$$ {\mbox{$m_{K+1}^{(1)}\leq m_{K}^{(1)}$ for all $K>1$. }}$$
The second finite minimization problem is the following: \begin{equation}\label{dirdiscreto2}
m^{(2)}_K:=\min \left\{ \sum_{j=-K}^K \left(\frac{1}{2r^2} (z_{j+1}-z_{j})^2+ W(z_j)\right)
\right\},
\end{equation}
where the class of competitors $(z_j)$ is such that   $z_{j}=-z_{-j-1}$ and $z_{j}=1$ for all~$j\ge K-1$. 

Choosing the competitor
$w_j:=1$ for all $j\geq 0$, we get that
$$ {\mbox{$m^{(2)}_K\leq \displaystyle\frac{2}{r^2}$ for all $K>1$.}}$$
Moreover, we see that
$$ {\mbox{$m_{K+1}^{(2)}\leq m_{K}^{(2)}$ for all $K>1$.}}$$
Arguing as in the proof of Lemma \ref{competit}, one can
show that
monotone sequences lower the energy,
and therefore, as in Proposition~\ref{exdir}, one obtains that,
for all $K>2$, there exist monotone nondecreasing solutions $(w^K_j)$ and $(z^K_j)$ respectively to the minimization problem \eqref{dirdiscreto1} and \eqref{dirdiscreto2}.  

Let us now fix  $(\phi_j)$ such that $\phi_0:=0$,  $\phi_j=-\phi_{-j}$ and $\phi_{j}:=0$ for all~$j\ge K-1$. So, 
the sequence $w_j^K+\delta \phi_j$ is an admissible competitor
for the minimization problem~\eqref{dirdiscreto1}
for all $\delta\in\R$. Therefore, by the minimality of $w_j^K$,
we conclude that 
\begin{equation}\label{FGYUSCVBHJ} \sum_{j=-K+1}^{K-1} \left(\frac{1}{r^2} (w^K_{j+1}-2w^K_{j}+w^K_{j-1})- W'(w^K_j) \right)\phi_j=0.\end{equation}
Now, fixed~$j_*\in (0, K-1)$, we choose $\phi_j$ such that $\phi_{j_*}:=1$ (and so  by assumption  $\phi_{-j_*}=-1$) and $0$ elsewhere. 
Substituting in~\eqref{FGYUSCVBHJ}, we thereby get that, for all $j\in (0, K-1)$, 
 \[\left(w^K_{j+1}-2 w^K_{j}+w^K_{j-1}-r^2 W'(w^k_{ j})\right)- \left(w^K_{- j+1}-2 w^K_{- j}+w^K_{- j-1}-r^2 W'(w^k_{- j})\right) =0.\] 
{F}rom this,
using the fact that $w^K_{j}=-w^K_{-j}$ and that $W'$ is an odd function on $[-1,1]$ by assumption \eqref{DOUBLEW}, we conclude that
\begin{equation}\label{eldis1}w^K_{j+1}-2 w^K_j+w^K_{j-1}= r^2 W'(w^k_j)\qquad \forall j\in (-K+1, K-1).\end{equation}
A similar argument gives that also $z^K_j$ satisfies \eqref{eldis1}, namely
\begin{equation}\label{eldis1 per z}z^K_{j+1}-2 z^K_j+z^K_{j-1}= r^2 W'(z^k_j)\qquad \forall j\in (-K+1, K-1).\end{equation}
Now, using the monotonicity of  $(w^K_j)$ and $(z^K_j)$, we can also take the limit as~$K\to+\infty$
for the sequences  $(w^K_j)$ and $(z^K_j)$.
In this way, we obtain two sequences that we denote by~$ (\bar w_j)$ and $(\bar z_j)$, respectively.

We notice that $(\bar w_i)$ and $(\bar z_i)$ are monotone nondecreasing. Also, they satisfy \eqref{1.13bis} and are odd, in the sense that 
$\bar w_0=0$ and $\bar w_n=-\bar w_{-n}$ for all $n>0$, whereas $\bar z_n=-\bar z_{-n-1}$. 
Moreover, using \eqref{eldis1} and~\eqref{eldis1 per z}, we get that they are solutions to \eqref{ricorsiva1}. This conclude the proof
of Proposition \ref{corodis}.\end{proof}

\begin{remark} Concerning the proof of Proposition~\ref{corodis},
we also observe that, arguing as in the proof of Theorem \ref{exthm}, we get that   
$(\bar w_j)$ is a solution to the minimization problem 
 \begin{equation}\label{dirdiscreto1fin}
\min \left\{ \sum_{j=-\infty}^{+\infty} \left(\frac{1}{2r^2} (w_{j+1}-w_{j})^2+ W(w_j)\right)
\right\}\end{equation} among all sequences such that  ~$w_j\in[-1,1]$ for all~$j\in\Z$, $w_0=0$,  $w_j=-w_{-j}$ and $\lim_{j\to +\infty} w_{j}=1$,
whereas $(\bar z_j)$ is a solution to the minimization problem 
 \begin{equation}\label{dirdiscreto2fin}
\min \left\{ \sum_{j=-\infty}^{+\infty} \left(\frac{1}{2r^2} (z_{j+1}-z_{j})^2+ W(z_j)\right)
\right\}\end{equation} among all sequences such that  ~$z_j\in[-1,1]$ for all~$j\in\Z$, $z_i=-z_{-i-1}$ for all $i\geq 0$ and  $\lim_{j\to +\infty} z_{j}=1$. 
\end{remark}

Now, we prove
Corollary \ref{coro1}, as a consequence of Proposition \ref{corodis}. 
 \begin{proof}[Proof of Corollary \ref{coro1}] 
Let  $(\bar w_n)_n$ and  $(\bar z_n)_n$
be as in Proposition \ref{corodis}. 
We define
$$ u(x):=\bar w_n\qquad{\mbox{ and }}\qquad
v(x):=\bar z_n\qquad{\mbox{for all }}x\in [2nr, 2nr+2r).$$
Then, $u$ and~$v$  are  monotone nondecreasing,
satisfy \eqref{LIMITS} and, in view of \eqref{ricorsiva1},
they are solutions to~\eqref{el1}, as desired.

We stress that~$u$ and~$v$ are geometrically different (namely,
they are not equal up to a translation). Indeed, if we define~$I_n:=[2nr, 2nr+2r)$,
we see that
\begin{equation}\label{DIVE}
{\mbox{$u=0$ on an odd number of intervals~$I_n$,
and $v=0$ on an even number of intervals~$I_n$.}}
\end{equation}
To check this, we first observe that~$u=\bar w_0=0$ in~$I_0$.
By monotonicity, we can take~$\bar n$ to be the greatest~$n$
for which~$u=0$ in~$I_n$, and then~$u=0$ in~$I_0\cup\dots I_{\bar n}$,
with~$u>0$ in~$I_{\bar n+1}$. Then, recalling~\eqref{PARDI}, it follows that~$u=0$
in~$I_{-1}\cup\dots I_{-\bar n}$,
with~$u<0$ in~$I_{-\bar n-1}$. This says that~$u=0$
in~$I_0\cup I_{\pm 1}\dots I_{\pm\bar n}$ and~$u\ne0$ elsewhere.

Similarly, suppose that~$v=0$ in some interval~$I_j$. {F}rom~\eqref{PARDI},
we conclude that~$v=0$ also in~$I_{-j-1}$. Since~$j$ cannot be equal to~$-j-1$,
this argument always provides a couple of intervals on which~$v=0$.
The proof of~\eqref{DIVE} is thereby complete.
\end{proof} 

\section{The Dirichlet problem for $D_r$, and proofs of
Theorem~\ref{DIRI}, and of Corollaries \ref{COROCONT} and~\ref{costanteDIR}}\label{secdir} 

In this section we provide the proofs of
Theorem~\ref{DIRI}, of Corollary \ref{COROCONT} and of Corollary~\ref{costanteDIR}.

First of all, we observe that uniqueness of solutions  to \eqref{DS} is a direct consequence of the following
Maximum Principle for the operator $D_r$. 
\begin{lemma}[Maximum Principle for $D_r$]\label{MAX PLE}
Let~$r>0$, $a<b$, 
and~$u\in L^\infty_{{\rm{loc}}}(\R)$
be such that
\begin{equation}\label{9KUWEQUA} \left\{\begin{matrix}
D_r u\ge0&{\mbox{ in }}(a,b),\\
u\le0&{\mbox{ in }}[a-r,a],\\
u\le0&{\mbox{ in }}[b,b+r].
\end{matrix}\right.\end{equation}
Then~$u\le0$  almost everywhere in~$[a-r,b+r]$.
\end{lemma}

\begin{proof} By contradiction, let us assume that
$$ \sigma:=\sup_{ [a-r,b+r] } u >0.$$
Let
\begin{equation} \label{DEFS}
\begin{split}
{\mathcal{S}}\,:=\,&\big\{ 
\bar x\in [a-r,b+r] {\mbox{ s.t. there exists a sequence $x_k\in[a-r,b+r]$}}\\
&{\mbox{such that
$x_k\to \bar x$, $x_k$ are density points for $u$,  and~$u(x_k)\to\sigma$ as $k\to+\infty$}}
\big\}.\end{split}\end{equation}
Since~${\mathcal{S}}\subseteq[a-r,b+r]$, we can define
\begin{equation} \label{DEFS2} x^\star:=\sup {\mathcal{S}}\in [a-r,b+r].\end{equation}
We claim that, in fact, this supremum is attained, and
\begin{equation}\label{9eofhi33848}
x^\star=\max {\mathcal{S}}.\end{equation}
To check this, fix~$m\in\N$, to be taken as large as we wish.
By~\eqref{DEFS2}, we know that there exists~$\bar x_m\in {\mathcal{S}}$
such that~$|x^\star-\bar x_m|\le 1/m$.
Then, by~\eqref{DEFS}, we know that there exists
a sequence~$x_{k,m}\in[a-r,b+r]$ such that
$$ \lim_{k\to+\infty} x_{k,m}= \bar x_m
\qquad{\mbox{and}}\qquad\lim_{k\to+\infty} u(x_{k,m})=\sigma.$$
In particular, there exists~$k_m\in\N$ such that if~$k\ge k_m$ then
$$ |x_{k,m}- \bar x_m|\le\frac1m
\qquad{\mbox{and}} \qquad| u(x_{k,m})-\sigma|\le\frac1m.$$
Let now~$x^\star_m:=x_{k_m,m}$. By construction,
$$ |x^\star_m-x^\star|\le
|x_{k_m,m}- \bar x_m|+|\bar x_m-x^\star|\le\frac2m
\qquad{\mbox{and}} \qquad| u(x^\star_m)-\sigma|\le\frac1m.$$
This and~\eqref{DEFS} imply that~$x^\star\in{\mathcal{S}}$, which,
combined with~\eqref{DEFS2}, gives~\eqref{9eofhi33848}, as desired.

We also claim that
\begin{equation}\label{9eofhi33848BIS}
x^\star\in(a,b).\end{equation}
To check this, suppose, by contradiction, that~$x^\star\in[a-r, a]$ (the
case~$x^\star\in[ b,b+r]$
is similar). Then it must be that
\begin{equation}\label{9eofhi33848TRIS}
x^\star=a.\end{equation}
Indeed, if~$x^\star\in[a-r,a)$
and~$x_k\to x^\star$ as~$k\to+\infty$, we have that~$x_k\in[a-r,a)$
for large~$k$, and thus
$$ \sigma=\lim_{k\to+\infty} u(x_k)\leq 0,$$
which is a contradiction.

Then, in view of~\eqref{9eofhi33848TRIS}, we can take
a sequence~$y_k>a$ such that
$$ \lim_{k\to+\infty} y_k=a=x^\star\qquad{\mbox{and}}\qquad\lim_{k\to+\infty}u(y_k)=\sigma.
$$
So  we can exploit~\eqref{9KUWEQUA} at the point~$y_k$, for $k$ sufficiently large such that $y_k-r<a$, 
and write that
\begin{eqnarray*}
0 &\leq & r^2 D_r u(y_k)\\
&=& u(y_k+r)+u(y_k-r)-2u(y_k)\\
&\le& u(y_k+r)-2u(y_k)\\
&\le& \sigma-2u(y_k).
\end{eqnarray*}
Then, passing to the limit as~$k\to+\infty$, we obtain that~$0\le\sigma-2\sigma=-\sigma$,
which is a contradiction. The proof of~\eqref{9eofhi33848BIS} is thereby complete.

Hence, in light of~\eqref{9eofhi33848} and~\eqref{9eofhi33848BIS}, we can
take a sequence~$x_k\to x^\star$ with~$u(x_k)\to u(x^\star)$
as~$k\to+\infty$ and
exploit~\eqref{9KUWEQUA} at the point~$x_k$. In this way,
we find that
$$ 0\leq  u(x_k+r)+u(x_k-r)-2u(x_k)\le u(x_k+r)+\sigma-2u(x_k),$$
and therefore
$$ 0\le \lim_{k\to+\infty} \Big( u(x_k+r)+\sigma-2u(x_k) \Big)
=\lim_{k\to+\infty} u(x_k+r)-\sigma.$$
Consequently, writing~$\tilde x_k:=x_k+r$ and~$\tilde x:=x^\star+r>x^\star$,
we see that~$\tilde x_k\to\tilde x$ and~$u(\tilde x_k)\to\sigma$
as~$k\to+\infty$. This gives that~$\tilde x\in{\mathcal{S}}$, which is
in contradiction with~\eqref{9eofhi33848}.
\end{proof}

\begin{proof}[Proof of Theorem~\ref{DIRI}.]
The uniqueness statement plainly follows from Lemma~\ref{MAX PLE}, hence
we focus on the proof of the fact that the function in \eqref{SOLUZIONE} is a solution to ~\eqref{DS}, which is a direct, albeit tricky, computation.
To this end, we observe that, by~\eqref{DK},
\begin{equation}\label{DKK}\begin{split}&
\overline k(x+r)=\overline k(x)-1,\qquad
\overline k(x-r)=\overline k(x)+1,\qquad\\&
\underline k(x+r)=\underline k(x)+1\qquad
\underline k(x-r)=\underline k(x)-1,\end{split}
\end{equation}
thus
\begin{equation}\label{DKKBIS}
\overline k(x\pm r)+\underline k(x\pm r)=\overline k(x)+\underline k(x)
,\end{equation}
and, moreover,
\begin{equation}\label{DKKK}
\begin{split}&
{\mbox{if~$y\in(a-r,a]$ and~$z\in[b,b+r)$,}}\\&{\mbox{then
$\overline k(y),\underline k(z)\in\left[\displaystyle\frac{b-a}r,\,\displaystyle\frac{b-a}r+1\right)$,
and~$\underline k(y)=0=\overline k(z)$.}}\end{split}
\end{equation}
Now, to check~\eqref{DS}, fixed~$x\in(a,b)$, we need to distinguish four cases:
\begin{eqnarray}
&& \label{CASO1} x+r\in[b,b+r) {\mbox{ and }} x-r\in(a-r,a],\\
&& \label{CASO2} x+r\in(a,b) {\mbox{ and }} x-r\in(a-r,a],\\
&& \label{CASO3} x+r\in[b,b+r) {\mbox{ and }} x-r\in(a,b)\\{\mbox{and }}&&
\label{CASO4} x+r\in(a,b) {\mbox{ and }} x-r\in(a,b).
\end{eqnarray}
We start by taking into account the case in~\eqref{CASO1}.
Then, by~\eqref{DKKK}, we have that~$\underline k(x-r)=0=\overline k(x+r)$,
and so, by~\eqref{DKK}, we get that~$\underline k(x)=1=\overline k(x)$,
therefore
\begin{eqnarray*}&&
u(x+r)+u(x-r)-2u(x)\\&=&\beta(x+r)+\alpha(x-r)-\Bigg(
\alpha(x-r)+
\beta(x+r) 
-r^2\,f(x)\Bigg)\\
&=& r^2\,f(x),
\end{eqnarray*}
which proves~\eqref{DS} in this case.

Let us now suppose that~\eqref{CASO2} holds true.
Then, by~\eqref{DKK} and~\eqref{DKKK},
we have that~$\underline k(x)=1+
\underline k(x-r)=1$ and consequently, recalling also~\eqref{DKKBIS},
\begin{eqnarray*}&&
u(x+r)+u(x-r)-2u(x)\\&=&
\frac{1}{\overline k(x)+1}\,\Bigg(
(\overline k(x)-1)\,\alpha(x-r)+
2\beta(x+\overline k (x)r)
-r^2(\overline k(x)-1)\,f(x) \\ &&\qquad
-2r^2\,\sum_{j=1}^{\overline k(x)-2} j f\big(x+(\overline k(x)-j)r\big) 
-2r^2(\overline k(x)-1)\,f(x+r)\Bigg)
+\alpha(x-r)\\&&\qquad-
\frac{2}{\overline k(x)+1}\,\Bigg(
\overline k(x)\,\alpha(x-r)+
\beta(x+\overline k (x)r)
-r^2\,\sum_{j=1}^{\overline k(x)-1} j f\big(x+(\overline k(x)-j)r\big) 
-r^2\overline k(x)\,f(x)\Bigg)\\&=&
\frac{1}{\overline k(x)+1}\,\Bigg(
-r^2(\overline k(x)-1)\,f(x) 
-2r^2\,\sum_{j=1}^{\overline k(x)-2} j f\big(x+(\overline k(x)-j)r\big) 
-2r^2(\overline k(x)-1)\,f(x+r)\\&&\qquad
+2r^2\,\sum_{j=1}^{\overline k(x)-1} j f\big(x+(\overline k(x)-j)r\big) 
+2r^2\overline k(x)\,f(x)\Bigg)\\
&=&
\frac{1}{\overline k(x)+1}\,\Bigg(
-r^2(\overline k(x)-1)\,f(x) 
-2r^2(\overline k(x)-1)\,f(x+r)
+2r^2(\overline k(x)-1)\,f(x+r) 
+2r^2\overline k(x)\,f(x)\Bigg)
\\ &=& r^2\,f(x).
\end{eqnarray*}
This proves~\eqref{DS} in this case,
and we now consider the case in~\eqref{CASO3}.
In this case, by~\eqref{DKK} and~\eqref{DKKK},
we have that~$\overline k(x)=1+
\overline k(x+r)=1$ and consequently, recalling also~\eqref{DKKBIS},
\begin{eqnarray*}&&
u(x+r)+u(x-r)-2u(x)\\&=&\beta(x+r)+
\frac{1}{1+\underline k(x)}\,\Bigg(
2\,\alpha(x-\underline k (x)r)+
(\underline k(x)-1)\,\beta(x+r)
-2r^2\,\sum_{j=1}^{\underline k(x)-2} j f\big(x-(\underline k(x)-j)r\big) \\&&\qquad
-r^2(\underline k(x)-1)\, f(x) 
-2r^2\,(\underline k(x)-1)\,f(x-r)\Bigg)
\\ &&\qquad-
\frac{2}{1+\underline k(x)}\,\Bigg(
\alpha(x-\underline k (x)r)+
\underline k(x)\,\beta(x+r)
-r^2\,\sum_{j=1}^{\underline k(x)-1} j f\big(x-(\underline k(x)-j)r\big) 
-r^2\,\underline k(x)\,f(x)\Bigg)\\&=&
\frac{1}{1+\underline k(x)}\,\Bigg(
-2r^2\,\sum_{j=1}^{\underline k(x)-2} j f\big(x-(\underline k(x)-j)r\big) 
-r^2(\underline k(x)-1)\, f(x) 
-2r^2\,(\underline k(x)-1)\,f(x-r)
\\ &&\qquad+
2r^2\,\sum_{j=1}^{\underline k(x)-1} j f\big(x-(\underline k(x)-j)r\big) 
+2r^2\,\underline k(x)\,f(x)\Bigg)\\&=&
\frac{1}{1+\underline k(x)}\,\Bigg(
-r^2(\underline k(x)-1)\, f(x) 
-2r^2\,(\underline k(x)-1)\,f(x-r)
+2r^2\,(\underline k(x)-1) \,f(x-r) 
+2r^2\,\underline k(x)\,f(x)\Bigg)\\&=& r^2\,f(x).
\end{eqnarray*}
This proves~\eqref{DS} in this case. So, it only remains to check~\eqref{CASO4}.
To this end,
\begin{eqnarray*}&&
u(x+r)+u(x-r)-2u(x)\\&=&
\frac{1}{\overline k(x)+\underline k(x)}\,\Bigg(
(\overline k(x)-1)\,\alpha(x-\underline k (x)r)+
(\underline k(x)+1)\,\beta(x+\overline k (x)r)
-r^2(\overline k(x)-1)\,\sum_{j=1}^{\underline k(x)} j f\big(x-(\underline k(x)-j)r\big) 
\\&&\qquad-r^2(\underline k(x)+1)\,\sum_{j=1}^{\overline k(x)-2} j f\big(x+(\overline k(x)-j)r\big)
-r^2(\overline k(x)-1)(\underline k(x)+1)\,f(x+r)\Bigg)\\&&\qquad+
\frac{1}{\overline k(x)+\underline k(x)}\,\Bigg(
(\overline k(x)+1)\,\alpha(x-\underline k (x)r)+
(\underline k(x)-1)\,\beta(x+\overline k (x)r)\\&&\qquad
-r^2(\overline k(x)+1)\,\sum_{j=1}^{\underline k(x)-2} j f\big(x-(\underline k(x)-j)r\big) 
-r^2(\underline k(x)-1)\,\sum_{j=1}^{\overline k(x)} j f\big(x+(\overline k(x)-j)r\big) \\
&&\qquad-r^2(\overline k(x)+1)(\underline k(x)-1)\,f(x-r)\Bigg)
\\&&\qquad-\frac{2}{\overline k(x)+\underline k(x)}\,\Bigg(
\overline k(x)\,\alpha(x-\underline k (x)r)+
\underline k(x)\,\beta(x+\overline k (x)r)
-r^2\overline k(x)\,\sum_{j=1}^{\underline k(x)-1} j f\big(x-(\underline k(x)-j)r\big) \\&&\qquad
-r^2\underline k(x)\,\sum_{j=1}^{\overline k(x)-1} j f\big(x+(\overline k(x)-j)r\big) 
-r^2\overline k(x)\,\underline k(x)\,f(x)\Bigg)
\\ &=&
\frac{1}{\overline k(x)+\underline k(x)}\,\Bigg(
-r^2(\overline k(x)-1)\,\sum_{j=1}^{\underline k(x)} j f\big(x-(\underline k(x)-j)r\big) 
\\&&\qquad-r^2(\underline k(x)+1)\,\sum_{j=1}^{\overline k(x)-2} j f\big(x+(\overline k(x)-j)r\big)
-r^2(\overline k(x)-1)(\underline k(x)+1)\,f(x+r)\\&&\qquad
-r^2(\overline k(x)+1)\,\sum_{j=1}^{\underline k(x)-2} j f\big(x-(\underline k(x)-j)r\big) 
-r^2(\underline k(x)-1)\,\sum_{j=1}^{\overline k(x)} j f\big(x+(\overline k(x)-j)r\big) \\
&&\qquad-r^2(\overline k(x)+1)(\underline k(x)-1)\,f(x-r)
+2r^2\overline k(x)\,\sum_{j=1}^{\underline k(x)-1} j f\big(x-(\underline k(x)-j)r\big) \\&&\qquad
+2r^2\underline k(x)\,\sum_{j=1}^{\overline k(x)-1} j f\big(x+(\overline k(x)-j)r\big) 
+2r^2\overline k(x)\,\underline k(x)\,f(x)\Bigg)
\\ &=&
\frac{1}{\overline k(x)+\underline k(x)}\,\Bigg(
-r^2(\overline k(x)-1)\,\sum_{j=1}^{\underline k(x)-2} j f\big(x-(\underline k(x)-j)r\big) 
\\&&\qquad-r^2(\overline k(x)-1)(\underline k(x)-1)\, f(x-r) 
-r^2(\overline k(x)-1)\,\underline k(x)\,f(x) 
\\&&\qquad-r^2(\underline k(x)+1)\,\sum_{j=1}^{\overline k(x)-2} j f\big(x+(\overline k(x)-j)r\big)
-r^2(\overline k(x)-1)(\underline k(x)+1)\,f(x+r)\\&&\qquad
-r^2(\overline k(x)+1)\,\sum_{j=1}^{\underline k(x)-2} j f\big(x-(\underline k(x)-j)r\big) 
-r^2(\underline k(x)-1)\,\sum_{j=1}^{\overline k(x)-2} j f\big(x+(\overline k(x)-j)r\big) 
\\&&\qquad-r^2(\underline k(x)-1)(\overline k(x)-1)\,f(x+r)
-r^2(\underline k(x)-1)\,\overline k(x)\,f(x)
\\
&&\qquad-r^2(\overline k(x)+1)(\underline k(x)-1)\,f(x-r)
+2r^2\overline k(x)\,\sum_{j=1}^{\underline k(x)-2} j f\big(x-(\underline k(x)-j)r\big) 
\\&&\qquad+2r^2\overline k(x)\,(\underline k(x)-1)\, f(x-r) 
+2r^2\underline k(x)\,\sum_{j=1}^{\overline k(x)-2} j f\big(x+(\overline k(x)-j)r\big) 
\\&&\qquad
+2r^2\underline k(x)\,(\overline k(x)-1)\, f(x+r) 
+2r^2\overline k(x)\,\underline k(x)\,f(x)\Bigg)\\&=&
\frac{1}{\overline k(x)+\underline k(x)}\,\Big(
-r^2(\overline k(x)-1)(\underline k(x)-1)\, f(x-r) 
-r^2(\overline k(x)-1)\,\underline k(x)\,f(x) 
\\&&\qquad
-r^2(\overline k(x)-1)(\underline k(x)+1)\,f(x+r)-r^2(\underline k(x)-1)(\overline k(x)-1)\,f(x+r)
-r^2(\underline k(x)-1)\,\overline k(x)\,f(x)
\\
&&\qquad-r^2(\overline k(x)+1)(\underline k(x)-1)\,f(x-r)+2r^2\overline k(x)\,(\underline k(x)-1)\, f(x-r) 
+2r^2\underline k(x)\,(\overline k(x)-1)\, f(x+r) 
\\&&\qquad+2r^2\overline k(x)\,\underline k(x)\,f(x)\Big)\\
&=& r^2\,f(x),
\end{eqnarray*}
which completes the proof of~\eqref{DS} in this case as well.
The proof of Theorem~\ref{DIRI} is thereby finished.
\end{proof} 

Finally, under the assumption that the functions $\alpha, \beta,f$ are continuous, we study the regularity of the solution to \eqref{DS} given by Theorem \ref{DIRI}. 
\begin{proof}[Proof of Corollary~\ref{COROCONT}]
We start with the proof of~\eqref{BOUND LINF}.
First of all, if~$x\in[a-r,a]$, then~$u(x)=\alpha(x)$,
which obviously fulfills the bound in~\eqref{BOUND LINF}.
Similarly, if~$x\in[b,b+r]$, we have that~$y:=x+a-b-r\in[a-r,a]$
and we
have that
$$ |u(x)|=|\beta(x)|\le |\alpha(y)|+|\alpha(y)-\beta(x)|
=|\alpha(y)|+|\alpha(y)-\beta(y-a+b+r)|,
$$
and the bound in~\eqref{BOUND LINF} also follows in this case.
Hence, it remains to prove~\eqref{BOUND LINF} when~$x\in(a,b)$.
In this case, we write~\eqref{SOLUZIONE} 
as
\begin{equation}\label{9620002-1193480}
\begin{split}
u(x)\,&=\alpha(x-\underline k (x)r)+
\frac{\underline k(x)}{\overline k(x)+\underline k(x)}\, 
 \,\big(\beta(x+\overline k (x)r)
-\alpha(x-\underline k (x)r)\big) 
\\&\qquad - \frac{r^2}{\overline k(x)+\underline k(x)}\, \left(
\overline k(x)\,\sum_{j=1}^{\underline k(x)-1} j f\big(x-(\underline k(x)-j)r\big) 
+\underline k(x)\,\sum_{j=1}^{\overline k(x)-1} j f\big(x+(\overline k(x)-j)r\big) 
+\overline k(x)\,\underline k(x)\,f(x)\right)
\end{split}
\end{equation}
and therefore
\begin{equation}\label{546371230x93}
\begin{split}
|u(x)|\,&\le \|\alpha\|_{L^\infty([a-r,a])}+
\frac{\underline{k}(x)}{\overline k(x)+\underline k(x)}\,\sup_{{p\in[a-r,a]}\atop{q\in[b,b+r]}}|\alpha(p)-\beta(q)|
\\& 
+ \frac{r^2}{\overline k(x)+\underline k(x)}\|f\|_{L^\infty([a,b])}
\left(\overline{k}(x) \frac{(\underline k(x)-1)\underline{k}(x)}{2} +\underline{k}(x) \frac{(\overline k(x)-1)\overline{k}(x)}{2}+
\overline{k}(x)  \underline{k}(x)\right)\\
&\le \|\alpha\|_{L^\infty([a-r,a])}+
\sup_{{p\in[a-r,a]}\atop{q\in[b,b+r]}}|\alpha(p)-\beta(q)|
+r^2
\,\|f\|_{L^\infty([a,b])}\frac{\overline{k}(x)  \underline{k}(x)}{2}.
\end{split}
\end{equation}
We also observe that, by~\eqref{DK},
\begin{equation}\label{9620002-119348} 
\overline k(x)\le \frac{b-x}{r}+1\le\frac{b-a}{r}+1\qquad
{\mbox{and }}\qquad
 \underline k(x)\le \frac{x-a}{r}+1\le\frac{b-a}{r}+1.\end{equation}
Therefore
$$ r^2 \frac{\overline{k}(x)  \underline{k}(x)}{2}\le  \frac{r^2}{2} \left(\frac{b-a+r}{r}\right)^2\leq (b-a)^2+ r^2.$$
Inserting this into~\eqref{546371230x93},
we thus obtain the bound in~\eqref{BOUND LINF}, as desired.

We also observe that the continuity of~$u$ outside~${\mathcal{J}}$
claimed in the statement of Corollary~\ref{COROCONT} follows directly
from~\eqref{SOLUZIONE} and so, to complete the proof
of Corollary~\ref{COROCONT}, it only remains to check~\eqref{JUMPATMOST}.
To this end, we observe that, in view of~\eqref{DK},
$$ \overline k(x)+\underline k(x)\ge\frac{b-x}r+\frac{x-a}r=\frac{b-a}r.$$
Using this, \eqref{9620002-1193480} and~\eqref{9620002-119348},
the claim in~\eqref{JUMPATMOST} follows
by the fact that~$\overline k$ and~$\underline k$ have at most unit jumps.
\end{proof}

\begin{proof}[Proof of Corollary~\ref{costanteDIR}]
As can be easily checked from~\eqref{SOLUZIONE}, the solution to \eqref{DS}   is given by 
\begin{equation}\label{functu}
u(x):= \left\{\begin{matrix}
0& {\mbox{ if }}x\in [-r,0],
\\
\frac{\underline{k}(x)}{\overline{k}(x)+\underline{k}(x)} & {\mbox{ if }}x\in (0,1),
\\
1&{\mbox{ if }}x\in [1,1+r].
\end{matrix}\right.\end{equation}
Observe that this function is discontinuous  at $x\in (0,1)$ such that $x\not\in r\N$. 
Indeed, recalling that $n=\frac{1}{r}\in \N$ and  $x\in (0,1)$, we observe that     
\begin{eqnarray*}
&& \underline{k}(x)= \left\lceil \frac{x}{r}\right\rceil=
\left\lceil nx\right\rceil\qquad {\mbox{ and}}\\
&&
\overline{k}(x)=\left\lceil \frac{1-x}{r}\right\rceil
=\left\lceil n-nx\right\rceil =\begin{cases}
n- nx= n-\underline{k}(x)=\frac{1}{r}-\underline{k}(x) & {\mbox{ if }}nx=
\frac{x}{r}\in \N,\\ 
n-\underline{k}(x)-1=\frac{1}{r}-1-\underline{k}(x) & {\mbox{ if }} nx=
\frac{x}{r}\not \in \N,
\end{cases} \end{eqnarray*}
where~$\lceil x\rceil$ denotes the smallest integer~$z$ such that~$x\le z$.

As a consequence,
\[\underline{k}(x)+\overline{k}(x)= \begin{cases} \frac{1}{r} & {\mbox{ if }}
x\in r\N,\\ 
\frac{1}{r}-1=\frac{1-r}{r} &{\mbox{ if }}  x\not\in r\N.  \end{cases}\]
Therefore the function $u$ defined in~\eqref{functu} can be written as
$$ u(x):= \begin{cases} 
0&{\mbox{ if }}x\in [-r,0],
\\ x &{\mbox{ if }} x\in (0,1) {\mbox{ and }}x\in r\N,\\  
\frac{r}{1-r}\left\lceil \frac{x}{r}\right\rceil &{\mbox{ if }} x\in (0,1)
{\mbox{ and }} x\not\in r\N,\\ 
1&{\mbox{ if }}x\in [1,1+r].
\end{cases}$$ 
Note  in particular  that $u(x)=0$ for all $x\in [-r, r)$, $u(x)=1$ for all $x\in (1-r, 1+r]$,
$0\leq u(x)\leq 1$ for all $x\in [-r, 1+r]$, and~$u$ has jump discontinuities
at the points of the form~$x=\frac{k}{n}$ with~$k\in\{1,\dots,n-1\}$.
\end{proof}

\end{document}